\documentclass[a4paper,11pt]{amsart}
\usepackage{amssymb}
\usepackage{amsfonts}
\usepackage{amsmath}
\usepackage{bbm}
\usepackage{xcolor}
\usepackage{xspace}
\usepackage{parskip}
\usepackage{graphicx}
\usepackage{subcaption}
\usepackage{hyperref}

\usepackage{url}
\usepackage{enumerate}
\usepackage{makecell}

\newtheorem{theorem}{Theorem}
\newtheorem{lemma}[theorem]{Lemma}

\newtheorem{definition}[theorem]{Definition}

\newtheorem{assumption}[theorem]{Assumption}

\newcommand{\expect}[1]{\mathbb{E}\left[  {#1} \right]}
\newcommand{\cexpect}[1]{\mathbb{E}\left[  {#1}
\Big|\mathcal{F}_{t_n}\right]}

\newcommand{\Dt}{{\Delta t}}
\newcommand{\Dtmin}{{\Delta t}_{\min}}
\newcommand{\Dtmax}{{\Delta t}_{\max}}
\newcommand{\Dtref}{{\Delta t}_{\text{ref}}}
\newcommand{\Xzero}{X_{\text{zero}}}

\newcommand{\DW}{\Delta W}

\newcommand{\Ito}{It\^o\xspace}

\newcommand{\Eonestep}{\mathrm{E}}
\newcommand{\Ey}[1]{E_Y({#1})}

\newcommand{\Cp}[1]{C_{#1}}

\newcommand{\be}{\begin{equation}}
\newcommand{\ee}{\end{equation}}

\begin{document}

\author{C\'onall Kelly}
\address{School of Mathematical Sciences, University College Cork, Ireland.}
\email{conall.kelly@ucc.ie}

\author{Gabriel J. Lord}
\address{
Department of Mathematics, IMAPP, Radboud University, Nijmegen, The Netherlands.}
\email{gabriel.lord@ru.nl}

\title{An adaptive splitting method for the Cox-Ingersoll-Ross process.}
\xdef\shorttitle{A splitting based method for CIR}

\date{\today}


\begin{abstract}
{We propose a new splitting method for strong numerical solution of the Cox-Ingersoll-Ross model. For this method, applied over both deterministic and adaptive random meshes, we prove a uniform moment bound and strong error results of order $1/4$ in $L_1$ and $L_2$ for the parameter regime $\kappa\theta>\sigma^2$. 
We then extend the new method to cover all parameter values by introducing a \emph{soft zero} region (where the deterministic flow determines the approximation) giving a hybrid type method to deal with the reflecting boundary. From numerical simulations we observe a rate of order $1$ when $\kappa\theta>\sigma^2$ rather than $1/4$. Asymptotically, for large noise, we observe that the rates of convergence decrease similarly to those of other schemes but that the proposed method  making use of adaptive timestepping displays smaller error constants. }
\end{abstract}


\maketitle

\section{Introduction}
We introduce a novel splitting method for the strong numerical solution of the Cox-Ingersoll-Ross (CIR) process, which typically arises as a model of stochastic volatility or for the pricing of interest rate derivatives in finance; see Cox et al.,~\cite{CoxIngersollRoss}. Similar equations arise in other application, for example  the modelling of changes in the membrane voltage of a neuron, see~\cite{Onofrio}. The CIR process is given by the \Ito-type stochastic differential equation (SDE) for $t\in[0,T]$
\begin{equation}\label{eq:CIR}
  dX(t) = \kappa\left(\theta - X(t) \right)dt + \sigma \sqrt{X(t)} dW(t),\quad \quad X(0)=X_0>0,
\end{equation}
where $W(t)$ is a Wiener process, and $\kappa$, $\theta$, and $\sigma$ are positive parameters.
Solutions of \eqref{eq:CIR} are almost surely (a.s.) non-negative. In general paths can achieve a value of zero but are reflected back into the positive half of the real line immediately. We propose a novel construct called the \emph{soft zero} to ensure dynamic consistency in the neighbourhood of this reflecting bounday. Further, when the Feller condition
\begin{equation}\label{eq:FellerCond}
2\kappa\theta>\sigma^2
\end{equation}
is satisfied solutions are a.s. positive.  

Although no analytic solution is available, $X(t)$ has 
a non-central chi-square distribution (conditional upon $X(s)$ for $0\leq s<t$), see \cite{Alfonsi2005,BroadieKaya2006,Glasserman,LordKoekkeokvanDijk2010}. 
For Monte Carlo estimation, exact sampling from this known conditional distribution is feasible but computationally inefficient and potentially restrictive if the Wiener process of \eqref{eq:CIR} is correlated, see Cozma \& Reisinger~\cite{Cozma}. 
As a consequence the numerical simulation of \eqref{eq:CIR} is an active topic of research, and techniques to handle the unbounded gradient of the diffusion coefficient near zero can be applied to more general equations. 

Two significant challenges for a numerical scheme for CIR are to preserve positivity and to prove convergence with as high a rate as possible over as large a parameter range as possible. Upper bounds on the order of strong convergence are most often shown to hold in a restricted parameter range; see Table \ref{tab:comparetheory}, which gives a summary of some of the key methods, the known rates of convergence and parameter ranges where the theoretical rates are valid. We now discuss two main approaches for the strong approximation of \eqref{eq:CIR} in the literature, the first is direct numerical simulation of \eqref{eq:CIR}, and the second is based on a Lamperti transformation.

For direct approximations of \eqref{eq:CIR} that preserve positivity of the numerical solution see, for example, the discussions in  \cite{Alfonsi2005,BD2004,HighamMao2005,Alfonsi2010,Alfonsi2013,BerkaouiBossyDiop2008,HefterHerzwurm2018}. The last of these represents a remarkable advance, showing that a broad class of Milstein-type methods over an equidistant mesh with a novel truncation in the neighbourhood of zero is convergent with rate across all parameter values. 

General results on the strong convergence of methods to approximate \eqref{eq:CIR} using equidistant evaluations of the Wiener process are in Hefter \& Jentzen~\cite{HefterJentzen2019}, with extension to a class of  path-dependent adaptive meshes in Hefter et al.,~\cite{HHMG2019}. For equidistant meshes Euler- and Milstein-type discretizations are included, and it is shown that such a method has a convergence order of at best $\delta/2$ where $0<\delta<2$ is the dimension of the squared Bessel process associated with CIR. 
In the general results of Hefter et al.,~\cite{HHMG2019}, when applied to \eqref{eq:CIR} with $\kappa\theta>2\sigma$ it is proved that adaptive algorithms are subject to the same upper and lower bounds on the rate of convergence as equidistant grids, though the error constant may be different. For larger relative values of $\sigma$ they state it is unclear whether adaptive algorithms can improve upon the optimal rate of convergence. In our work we examine adaptive discretizations of \eqref{eq:CIR} and observe the same rates of convergence with equidistant vs adaptive steps. Howevever, in numerical experiments adaptivity allows us to better capture the dynamics of the process and to reduce the error constant, hence improving efficiency.

In Section \ref{sec:numres} we compare our new splitting based method to several of these: the Milstein method of Hefter \& Herzwurm~\cite{HefterHerzwurm2018} and the fully truncated method proposed by Lord et al.,~\cite{LordKoekkeokvanDijk2010}. The latter method is widely used in practice and was shown in \cite{LordKoekkeokvanDijk2010} to be convergent in $L_1$ but without a rate (results of this type are common for these Euler variants, see for example Higham \& Mao~\cite{HighamMao2005}).
However recently Cozma \& Reisinger~\cite{Cozma} obtained strong order of convergence $1/2$ in $L_p$ under certain conditions. With $p=2$ this gives convergence with rate $1/2$ for $2\kappa\theta>3\sigma^2$. 

The second main approach is to apply a Lamperti transformation to \eqref{eq:CIR} and to numerically approximate the related process $Y=\sqrt{X}$. The SDE in $Y$ has additive noise but a drift coefficient with a singularity in the neighbourhood of zero (see \eqref{eq:Lamperti}).
We note that strong $L_2$ convergence for the transformed equation yields $L_1$ convergence for the original CIR process. 
A fully implicit Euler discretisation over a uniform mesh that preserves positivity of solutions was proposed in Alfonsi~\cite{Alfonsi2005} and shown to have uniformly bounded moments. A continuous time extension interpolating linearly between mesh points was shown to have strong $L_p$ order of convergence $1/2$ (up to a factor of $\sqrt{|\log(\Dt)|}$) in Dereich et al.,~\cite{DNS2012} when $2\kappa\theta>p\sigma^2$, a continuous-time variant based on the same implicit discretisation was shown to have strong $L_p$ convergence of order $1$ when $4\kappa\theta>3p\sigma^2$ in Alphonsi~\cite{Alfonsi2013}. In Chassagneux et al., \cite{CJM16} a variant which discretised the transformed SDE for $Y$ with an explicit projection method was shown to give strong convergence of between order $1/6$ and $1$ for $\kappa\theta>\sigma^2$ (see Table \ref{tab:comparetheory}). We consider in Section \ref{sec:numres} the implicit method of \cite{Alfonsi2005} and the projection method of \cite{CJM16}.

In this article, we construct and analyse a new positivity preserving numerical method for \eqref{eq:CIR} which is based on a splitting method applied to the Lamperti transform. This scheme falls outside the framework of Hefter \& Herzwurm~\cite{HefterHerzwurm2018} as it fails to satisfy their $L_p$-Lipschitz continuity requirement in the absence of an equivalent truncation strategy. As an example of a splitting strategy it merits an independent error analysis. Over a class of meshes that includes both uniform and adaptive meshes we prove a uniform moment bound for the numerical scheme, which leads to a strong $L_1$ convergence of order $\Dt^{1/4}$ for $\kappa\theta>\sigma^2$. 
These results together then allow us to prove strong $L_2$ convergence of the same order in that region. The parameter restriction arises from the need to control the first two inverse moments of the transformed SDE.

In common with other methods we observe in numerical experiments a far better rate of convergence for the proposed method than predicted by the analysis in the Feller regime. 
In fact, with increasing $\sigma$ we see higher rates up to when $\alpha:=(4\kappa\theta-\sigma^2)/8=0$.
The scheme can be used without modification for $\alpha \geq 0$.
For $\alpha<0$ we describe an extension of the adaptive version of the scheme in Section \ref{sec:alphaneg} that introduces the notion of a \emph{soft zero} region where the noise component is switched off and the unperturbed solution is solved exactly, giving a hybrid method. The splitting method is seen to be competitive across all values of $\alpha$ both in terms of the estimated rate of convergence and in that it consistently shows a small error constant. The {\em soft zero} approach means that for large time steps, by contrast with other methods, the solution is not simply projected to zero. We believe this to be a novel technique with  broad applicability.

The structure of the article is as follows. In Section~\ref{sec:prelim} we motivate the construction of the numerical scheme by describing the Lamperti transformed equation in variation of constants form, and set up the (potentially adaptive) mesh. In Section~\ref{sec:momBounds} we describe useful conditional moment and regularity bounds on the exact solution of the SDE \eqref{eq:CIR}. In Section \ref{sec:main} we present our three main results, all of which hold over a potentially adaptive mesh. First, we prove a uniform moment bound on a linear interpolant of the numerical scheme. Second, we prove an $L_2$-strong convergence result for the numerical discretisation of the Lamperti transformed scheme, which implies an $L_1$-strong convergence result for the untransformed scheme. This leads to the third main result, an $L_2$-strong convergence result for the untransformed scheme. In Section \ref{sec:alphaneg} we extend the scheme to the $\alpha<0$ regime and prove that numerical steps taken in the ``soft zero'' region satisfy an appropriate mean-square consistency bound. In Section \ref{sec:numres} we compare numerically our method to four methods, two from direct simulation of CIR and two derived from the Lamperti transformation.  

\begin{table}
\begin{tabular}{ |c|c|c|c| } 
 \hline
 Scheme &  Norm & \makecell{Parameter \\Range} & \makecell{Rate} 
 \\ 
 \hline \hline
 
 \makecell{Proposed Splitting \\Method} &  $L_1$ and $L_2$  & $\kappa\theta>\sigma^2$& $1/4$ \\ 
  \hline 
\makecell{Truncated Milstein\\ (Hefter \& Herzwurm\\ \cite{HefterHerzwurm2018})}& $L_p$ & no restriction & \makecell{$\frac{1}{2p}\wedge\frac{2\kappa\theta}{p\sigma^2}-\epsilon$\\
 }
\\ \hline
\makecell{Drift Implicit\\ Square-Root Euler\\ (Alphonsi~\cite{Alfonsi2013}) }  
& \makecell{$L_p$\\ $p\in[1,\frac{4\kappa\theta}{3\sigma^2})$ }& $\kappa\theta>(1\vee \frac{3}{4}p)\sigma^2$ & 1 \\ 
 \hline 
\makecell{Projected Euler\\ (Chassagneux et al.,\\ \cite{CJM16})} & $L_1$ & \makecell[c]{$\kappa\theta>\frac{5}{2}\sigma^2$ \\ $\kappa\theta>\frac{3}{2}\sigma^2$ \\ $\kappa\theta>\sigma^2$} & \makecell[c]{$1$\\$\frac{1}{2}$\\$(\frac{1}{6},\frac{1}{2}-\frac{\sigma^2}{2\kappa\theta+\sigma^2})$} \\ 
 \hline 
\makecell{Fully Truncated\\ (Cozma \& Reisinger\\ \cite{Cozma})} & 
  \makecell{$L_p$ \\ $p\in[2,\frac{2\kappa\theta}{\sigma^2} -1)$}& $\kappa\theta>\frac{3}{2}\sigma^2$ & 1/2\\ 
 \hline
\end{tabular}
\caption{A selection of results with non-logarithmic rates of convergence(see also \cite[Table~1]{BossyOlivero2018}. We compare numerical results for these methods in Section \ref{sec:numres}.}
\label{tab:comparetheory}
\end{table}

\section{Mathematical Preliminaries}\label{sec:prelim}
\subsection{The Cox-Ingersoll-Ross equation}
For $t\geq 0$, the integral equation of \eqref{eq:CIR} is given by
\begin{equation}
 X(t)  = X_0 + \kappa\int_0^t (\theta - X(s)) ds + \sigma \int_0^t \sqrt{X(s)} dW(s),
 \label{eq:CIRint}
 \end{equation}
which can be written in variation of constants form as
\begin{eqnarray}
X(t)&=& e^{-\kappa t }X_0 + \theta(1 - e^{-\kappa t}) + \sigma \int_0^t e^{-\kappa(t-s)}\sqrt{X(s)} dW(s).
     \label{eq:CIRVoC2}
\end{eqnarray}
If we suppose that Feller's condition \eqref{eq:FellerCond} holds
we can equivalently consider the SDE with additive noise yielded by a Lamperti transformation. Letting $Y(t)=\sqrt{X(t)}$ for all $t\in[0,T]$ and 
\begin{equation}\label{eq:alphabetagamma}
\alpha:=(4\kappa\theta - \sigma^2)/8, \quad \beta:=\kappa/2, \quad  \text{and}\quad \gamma:=\sigma/2, 
\end{equation}
we have, after an application of It\^o's formula, the SDE
\begin{equation}\label{eq:Lamperti}
dY(t) = (\alpha Y^{-1}(t) - \beta Y(t)) dt + \gamma dW(t),\quad t\geq 0.
\end{equation}
This in its turn may be written in variation of constants form as
\begin{equation}
\label{eq:LampVoC}
Y(t) = e^{-\beta(t-t_{n})} Y(t_n) + \int_{t_n}^t e^{-\beta(t-s)} \frac{\alpha}{Y(s)} ds + \gamma \int_{t_n}^t e^{-\beta(t-s)} dW(s).
\end{equation}
Note that Feller's condition as given in \eqref{eq:FellerCond} may be expressed as $\alpha>\sigma^2/8$ or equivalently $2\alpha>\gamma^2$.

We prove a uniform moment bound, as well as strong $L_2$-convergence with order at least $1/4$ of numerical approximations to both \eqref{eq:CIR} and to the transformed equation \eqref{eq:Lamperti} under the following assumption, which implies the Feller condition \eqref{eq:FellerCond}:
\begin{assumption}\label{as:2invMom}
Let $\alpha>3\sigma^2/8$, or equivalently
$\kappa\theta>\sigma^2.$
\end{assumption}

Numerically, we examine strong $L_2$-convergence using a deterministic (uniform) mesh under the following parameter set, which is a superset of that defined in Assumption \ref{as:2invMom}:
\begin{assumption}\label{as:alphaPos}
Let $\alpha\geq 0$, or equivalently
$4\kappa\theta\geq\sigma^2.$
\end{assumption}
Finally we numerically investigate strong $L_2$-convergence using a random adaptive mesh under the following complementary parameter set:
\begin{assumption}\label{as:alphaNeg}
Let $\alpha<0$, or equivalently
$4\kappa\theta<\sigma^2.$
\end{assumption}
Note that Feller's condition is necessary for Assumption \ref{as:2invMom}, and sufficient but not necessary for Assumption \ref{as:alphaPos}. If Assumption \ref{as:alphaNeg} holds, Feller's condition does not.

\subsection{The mesh}\label{sec:mesh}
Our method may be implemented on either a deterministic or adaptive (random) mesh.
Therefore we introduce here a generic nonuniform mesh with nodes that may be selected randomly subject to certain measurability requirements, with a view to proving strong convergence under a minimal set of constraints on the mesh.

Consider the mesh $\{t_0,t_1\ldots,t_N\}$ on the interval $[0,T]$, where $t_0=0$ and $t_N=T$. Points on the mesh are assumed to be distinct and may be selected in a way that may or may not be path dependent. For example the mesh may be deterministic (e.g. with uniform steps across trajectories) or random (e.g. in an adaptive manner by choosing $t_{n+1}$ based upon the observed value of a discretisation at $t_n$). In the latter case $N$ will be an $\mathbb{N}$-valued random variable, and in either case, we denote $\Dt_{n+1}:=t_{n+1}-t_n$ for all $n=0,\ldots,N-1$. 
\begin{definition}\label{def:fil}
Let $(\mathcal{F}_t)_{t\geq 0}$ be the natural filtration of $W(t)$. Suppose that each member of  $\{t_n:=\sum_{i=1}^{n}\Dt_{i}\}_{n\in\{0,\ldots,N\}}$, with $t_0=0$, is an $\mathcal{F}_t$-stopping time. That is to say $\{t_n\leq t\}\in\mathcal{F}_t$ for all $t\geq 0$.
If $\tau$ is any $\mathcal{F}_t$-stopping time $\tau$ then (see \cite{mao2006stochastic})
\[
\mathcal{F}_{\tau}:=\{B\in\mathcal{F}\,:\,B\cap\{\tau\leq t\}\in\mathcal{F}_t,\,\text{ for all }\,t\geq 0\}.
\]
This allows us to condition on $\mathcal{F}_{t_n}$ at any point on the time-set $\{t_n\}_{n\in\{0,\ldots,N\}}$.
\end{definition}
All our theoretical results apply on a mesh that satisfies the following: 
\begin{assumption}\label{as:DtMeas}
$\Dt_{n+1}$ is $\mathcal{F}_{t_n}$-measurable and $N<\infty$ a.s. Moreover there exists a deterministic constant $\Dtmax$ such that $\Dt_n\leq\Dtmax$ for all $n=1,\ldots,N$.
\end{assumption}
For our main convergence results we must additionally assume the following stronger conditions on the mesh:
\begin{assumption}\label{as:Nbounded}
For each fixed $\Dtmax$, there exists
\begin{itemize}
    \item[A.] a deterministic integer $N_{\max}<\infty$ such that $N\leq N_{\max}$ a.s; 
    \item[B.] a constant $T_{\max}$, independent of $\Dtmax$ and $N_{\max}$, such that 
    \begin{equation}\label{eq:TmaxDef}
        \Dtmax N_{\max}\leq T_{\max}. 
    \end{equation}
\end{itemize}
\end{assumption}
Part B holds if, for example, there exists a minimum stepsize $\Dtmin$ held in a fixed ratio to $\Dtmax$: see \cite{KL2018_1,KLM2020}. Whereas, Part A by itself, only requires that the number of steps taken over the entire interval of simulation has a deterministic upper bound. Moreover, both of Assumptions \ref{as:DtMeas} and \ref{as:Nbounded} are automatically satisfied if the mesh is constructed deterministically. 
Only Part A of Assumption \ref{as:Nbounded} is required to prove $L_2$-convergence directly for the transformed equation \eqref{eq:Lamperti}, whereas Parts A and B are both required to achieve our $L_2$-convergence result for \eqref{eq:CIR}. 

If an adaptive timestepping strategy is used the a.s. finiteness (Assumption \ref{as:DtMeas}) or boundedness (Assumption \ref{as:Nbounded}) of $N$ would need to be confirmed.

\begin{definition}\label{def:Nt}
For each $t\in [0,T]$, define the (potentially random) integer $N^{(t)}$ such that
\[
N^{(t)}:=\max\{n\in\mathbb{N}\setminus\{0\}\,:\,t_{n-1}<t\}.
\]
Set $N:=N^{(T)}$ and $t_N:=T$, so that $T$ is always the last point on the mesh.
\end{definition}
For any $t\in [0,T]$, $N^{(t)}$ is almost everywhere (a.e.) the index of the right endpoint of the step that contains $t$, and by construction and Assumption \ref{as:DtMeas}, we have $N^{(t)}\leq N<\infty$ a.s, or $N^{(t)}\leq N\leq N_{\max}<\infty$ a.s. if Assumption \ref{as:Nbounded} Part A additionally holds.

If the mesh is constructed adaptively, $\Delta W_{n+1}:=W(t_{n+1})-W(t_n)$ is a Wiener increment over a random interval the length of which depends on $X_n$, through which it depends on $\{W(s),\,s\in[0,t_n]\}$. Therefore $\Delta W_{n+1}$ will not be independent of $\mathcal{F}_{t_n}$; indeed it is not necessarily normally distributed. Since $\Dt_{n+1}$ is a bounded $\mathcal{F}_{t_n}$-stopping time and $\mathcal{F}_{t_n}$-measurable, then $W(t_{n+1})-W(t_n)$ is $\mathcal{F}_{t_n}$-conditionally normally distributed, by Doob's optional sampling theorem (see for example Shiryaev~\cite{Shiryaev96})
\begin{eqnarray*}
\mathbb{E}[W(t_{n+1})-W(t_n) |\mathcal{F}_{t_n}]&=&0,\quad a.s.;\\
\mathbb{E}[|W(t_{n+1})-W(t_n)|^2 |\mathcal{F}_{t_n}]&=&\Dt_{n+1},\quad a.s.
\end{eqnarray*}

\subsection{The splitting method}
We start by considering the transformed SDE \eqref{eq:Lamperti}. The main approximation over a single step from $t_j$ to $t_{j+1}$ is based on the exact solution of the ODE
$$\frac{dz(t)}{dt}  = \alpha z^{-1}(t),\quad t\in[t_j,t_{j+1}],$$
which,  when $z(t_j)$ is given, can be computed to be
$$z(t) = \sqrt{z(t_j)^2 + 2\alpha(t-t_j)},\quad t\in[t_j,t_{j+1}].$$
This we combine with the exponential integrator based approximation of the OU-type SDE 
\begin{equation*}
dx(t) = - \beta x(t) dt + \gamma dW(t),\quad t\in[t_j,t_{j+1}].
\end{equation*}
The method that results is equivalent to the Lie-Trotter composition of the exact flows of the following subequations
\begin{align*}
dY^{[1]}(t)&=\alpha\left(Y^{[1]}(t)\right)^{-1}dt,\quad Y^{[1]}(0)=Y_0^{[1]}\\
dY^{[2]}(t)&=\gamma dW(t),\quad Y^{[2]}(0)=Y_0^{[2]};\\
dY^{[3]}(t)&=-\beta Y^{[3]}(t)dt,\quad Y^{[3]}(0)=Y_0^{[3]}.
\end{align*}
Thus we obtain the approximation $Y_n$ to $Y(t_n)$ of \eqref{eq:Lamperti}
\begin{equation}\label{eq:DiscLamperti}
Y_{n+1} = e^{-\beta \Dt_{n+1}}\left(\sqrt{(Y_n)^2+2\alpha\Dt_{n+1}}+\gamma\DW_{n+1}\right),
\end{equation}
where $\DW_{n+1}=W(t_{n+1})-W(t_n)$. Note that changing the order of the splitting results in a method where we can not readily control the timestep to preserve positivity.

Our analysis requires a continuous form of this intermediate scheme defined over $t\in[t_n,t_{n+1}]$ given by
\begin{equation}
\label{eq:CtsLamp}
\bar{Y}(t) = e^{-\beta(t-t_n)}\left(Y_n + \int_{t_n}^t \frac{\alpha}{\sqrt{(Y_n)^2+2\alpha(s-t_n)}} ds\right) + e^{-\beta(t-t_n)} \gamma \int_{t_n}^t dW(s).
\end{equation}
It is straightforward to confirm that $\bar{Y}(t_n) = Y_n$ and $\bar{Y}(t_{n+1}) = Y_{n+1}$.

The scheme for \eqref{eq:CIR} is then defined by $X_{n+1} := (Y_{n+1})^2$, so
\begin{equation}\label{eq:schemeSq}
X_{n+1} =e^{-2\beta \Dt_{n+1}}\left(\sqrt{X_n+2\alpha\Dt_{n+1}}+\gamma\DW_{n+1}\right)^2,
\end{equation}
or equivalently 
\begin{multline}
\label{eq:scheme1step}
X_{n+1} = e^{-\kappa \Dt_{n+1}}\left(X_n+2\alpha\Dt_{n+1}\right.\\ 
\left.+ \sigma \sqrt{X_n+2\alpha\Dt_{n+1}} \left(\DW_{n+1}\right) + \sigma^2 \frac{(\DW_{n+1})^2}{4} \right). 
\end{multline}
An immediate consequence of the construction is that values of the numerical solution $\{X_n\}_{n=0}^N$ are non-negative for any $\alpha>0$.

Our results easily extend to a Strang-like splitting which results in only one extra term that is of higher order, and is given by
\[
X_{n+1}=e^{-2\beta\Dt_{n+1}}\left(\sqrt{X_n+\alpha\Dt_{n+1}}+\gamma\DW_{n+1}\right)^2+\alpha\Dt_{n+1}.
\]
%
%
%
We did not observe numerical evidence of significant advantage. Nonetheless there is scope for further investigation in this direction.

\subsection{Extension to $\alpha<0$: adaptivity and a soft zero}
\label{sec:alphaneg}
The scheme given by \eqref{eq:schemeSq} contains the expression $\sqrt{X_n+2\alpha\Dt_{n+1}}$. When $\alpha>0$ this term is real and strictly positive for all $\Dt_{n+1}>0$. To ensure that the square root is real and positive when $\alpha \leq 0$ we can adapt the time step $\Dt_{n+1}$ to impose $X_n+2\alpha\Dt_{n+1} >0$ by taking
\begin{equation}
\label{eq:adaptdt}
  \Dt_{n+1} = \min\left\{0.95\frac{X_n}{2|\alpha|},\Dtmax\right\}.
\end{equation}
However, it is insufficient to simply apply the scheme \eqref{eq:schemeSq} over the adaptive mesh \eqref{eq:adaptdt} for two reasons. First, since Feller's condition does not hold, the boundary at zero may be achieved in this parameter regime, and in this case adaptivity is insufficient to maintain strict positivity of $\Dt_{n+1}$. Second, our scheme is based on the transformed SDE \eqref{eq:Lamperti} which is only well defined if solutions cannot attain zero (i.e. if Feller's condition holds). Nonetheless the scheme itself is well defined if Feller's condition is violated, but we must in that case treat numerical solutions in a neighbourhood of zero carefully.

We are thus motivated to introduce a {\em soft zero} region $[0,\Xzero]$ for some $\Xzero>0$ so that, when $X_n\in[0,\Xzero]$, we approximate \eqref{eq:CIR} by the deterministic ODE on $t\in[t_n,s]$, given $u(t_n)$
$$ \frac{du}{dt} = \kappa(\theta-u), \quad \text{with solution} \quad u(t) = e^{-\kappa(t-t_n)}u(t_n)+ \theta\left(1-e^{-\kappa(t-t_n)} \right).$$
We construct $\Xzero$ by a rescaling of $u(t_n+\Dtmax)|_{u(t_n)=0}$: for any $\rho>1$ (we choose $\rho=2$ in our numerical experiments) define
\begin{equation*}
   \Xzero:=\rho^{-1}\theta(1-e^{-\kappa\Dtmax}) \leq \rho^{-1}2\kappa\theta \Dtmax.
\end{equation*}
Note that as $\Dtmax\to 0$, $\Xzero\to 0$. 
When $X_n<\Xzero$ we set 
\begin{equation}
\Dt_{n+1}=-\frac{1}{\kappa}\log\left(\frac{\Xzero-\theta}{X_n-\theta}\right)
\label{eq:dtsoftzero}
\end{equation}
so that $X_{n+1}=\Xzero$.

Thus, when the numerical solution enters the {\em soft zero} region from a step of the splitting method \eqref{eq:scheme1step}, we take a single step of length $\Dt_{n+1}$ computed according to \eqref{eq:dtsoftzero} after which we are guaranteed to have exited the {\em soft zero} region. The scheme then reverts to \eqref{eq:scheme1step} unless the numerical solution again drops below $\Xzero$. In this way we can preserve the drift dynamics of the underlying SDE when numerical solutions are close to zero, by contrast with a truncation or projection approach.

In Section \ref{sec:LEBSZ}, Lemma \ref{lem:softzeroLocalError}, we prove a local error estimate for the scheme when it operates in the \emph{soft zero} region. However, further work is required to prove global convergence. There is scope to extend the application of this type of \emph{soft zero} to preserve domain invariance for other SDEs in future work.

\section{Moment bounds for the continuous-time equation}\label{sec:momBounds}
The bounds in this subsection hold under either of Assumptions \ref{as:alphaPos} and \ref{as:alphaNeg}.
We start by providing $L_p$ bounds, conditional at $\mathcal{F}_{t_n}$. Lemma 2.1 in Bossy \& Diop \cite{BD2004} gives bounds of the form:
\begin{equation}
\label{eq:BossyDiopMomentBound}
\expect{\sup_{t\in[0,T]}X(t)^{2p}}\leq M_{2p}:=\Cp{1}(1+X(0)^{2p}),
\qquad p\geq 1,
\end{equation}
regardless of the parameter values. 
There is a natural extension to conditional moment bounds. For example we have a conditional bound following from the mean reverting property of solutions of \eqref{eq:CIR}.
\begin{lemma}\label{lem:MeanRevert}
Let $\left(X(t) \right)_{t \in [0,T]}$ be a solution of \eqref{eq:CIR}, let $t_n$ be a node on a (potentially random) mesh such that Assumption \ref{as:DtMeas} holds, and suppose that $t_n \leq t\leq T$. Then a.s.
\begin{eqnarray}
\cexpect{X(t)} &=& e^{-\kappa (t-t_n)} X(t_n) + \theta(1-e^{-\kappa (t-t_n)}) \label{eq:PREMeanRevert}\\
&\leq& X(t_n)+\theta.
\label{eq:MeanRevert}
\end{eqnarray}
\end{lemma}
\begin{proof}
Equation \eqref{eq:PREMeanRevert} is well-known and may be found in \cite{CoxIngersollRoss}. The inequality \eqref{eq:MeanRevert} follows.
\end{proof}

Furthermore we have the following: 
\begin{lemma}\label{lem:posMB}
Let $\left(X(t) \right)_{t \in [0,T]}$ be a solution of \eqref{eq:CIR}, and let $0 \leq t_n\leq T$. For any $X(0)>0$ and any $p>0$, there exist constants
$M_{1,p}<\infty$,
such that
\begin{equation}\label{eq:equCondMp1}
\mathbb{E} \left[ \sup_{u\in[t_n,T]}X(u)^{p} \biggl|\mathcal{F}_{t_n}\right]\leq M_{1,p}(1+X(t_n)^p),\quad a.s.
\end{equation}
Furthermore let $\left( Y(t) \right)_{t \in [0,T]}$ be a solution of \eqref{eq:Lamperti}, where Assumption \ref{as:2invMom} holds,
\begin{equation}\label{eq:equCondMp2}
\mathbb{E} \left[ \sup_{u\in[t_n,T]}Y(u)^{p} \biggl|\mathcal{F}_{t_n}\right]\leq M_{1,p}(1+Y(t_n)^p),\quad a.s.
\end{equation}
\end{lemma}
\begin{proof}
See Lemma 3 in \cite{KLM2020}.
\end{proof} 

We also require bounds on the first two inverse conditional moments of \eqref{eq:Lamperti}, see \cite{KLM2020} for proof.
\begin{lemma}\label{lem:fin}
Let $\left( Y(t) \right)_{t \in [0,T]}$ be a solution of \eqref{eq:Lamperti}, where Assumption \ref{as:2invMom} holds, and let $0 \leq t_n< s \leq T$. For any $Y(0)>0$, and 
for $p=1,2$, there exists $\Cp{2}(p,T)>0$ such that
\begin{equation} \label{eq:equM-2}
\mathbb{E} \left[ \frac{1}{Y(s)^{p}} \biggl| \mathcal{F}_{t_n} \right] \leq \frac{\Cp{2}(p,T)}{Y(t_n)^p}, \quad a.s.
\end{equation}
\end{lemma}

The following lemma characterises the conditional H\"older continuity of solutions of \eqref{eq:Lamperti}, and is a special case of Lemma 13 in \cite{KLM2020}.
\begin{lemma}\label{lem:adap}
Let $\left( Y(t) \right)_{t \in [0,T]}$ be a solution of \eqref{eq:Lamperti}, suppose that Assumption \ref{as:2invMom} holds, and let $\lbrace t_n \rbrace_{n \in \mathbb{N}}$ be a random mesh such that each $t_n$ is an $\mathcal{F}_t$-stopping time. Fix $n\in\mathbb{N}$ and suppose that $t_n \leq s \leq T$.
Then
\begin{equation}\label{eq:Yreg}
\mathbb{E} \left[ | Y(s)-Y(t_n)|^2 \big|\mathcal{F}_{t_n} \right] \leq 4\gamma^2 |s-t_n| + \bar L_{n} |s-t_n|^2, \ a.s,
\end{equation}
where
\begin{equation*}
\bar L_{n}:=2^{4} \left( \alpha^2 \frac{C_2(2,T)}{Y(t_n)^2} + |\beta|^2 M_{1,2}(1+Y(t_n)^2)\right)
\end{equation*}
is an $\mathcal{F}_{t_n}$-measurable random variable with finite expectation, and $C_2(2,T)$, $M_{1,2}$ are the constants defined by \eqref{eq:equM-2} and \eqref{eq:equCondMp2} in the statements of Lemmas \ref{lem:fin} and \ref{lem:posMB} respectively, setting $p=2$. 
\end{lemma}

\section{Main results: moment bounds and strong convergence}\label{sec:main}

\subsection{Moment bounds for the splitting scheme}
Our first main result is to prove a uniform moment bound for the scheme \eqref{eq:scheme1step} that applies over both deterministic and random meshes. 


\begin{theorem}\label{lem:SchemeMoment}
Let $\{X_n\}_{n\in\mathbb{N}}$ be a solution of \eqref{eq:scheme1step}, suppose that the (potentially random) mesh values $\{t_0,t_1,\ldots,t_N\}$ are selected  so that Assumption \ref{as:DtMeas} holds.
Then, where $N^{(t)}$ is as given in Definition \ref{def:Nt}, 
\[
\expect{X_{N^{(t)}}}\leq X_0+\kappa\theta T,\quad t\in[0,T].
\]
\end{theorem}
\begin{proof}
Fix $t\in [0,T]$ and let $N^{(t)}$ be as in Definition \ref{def:Nt}. From the form of \eqref{eq:schemeSq}, $X_n>0$, $n=0,\ldots,N^{(t)}$, a.s.
Take the expectation of both sides, conditional upon $\mathcal{F}_{t_n}$, to get 
\begin{eqnarray*}
\lefteqn{\cexpect{X_{n+1}}}\\
& =& e^{-\kappa\Dt_{n+1}}\left(X_n +2\alpha\Dt_{n+1}+\sigma\sqrt{X_n+2\alpha\Dt_{n+1}}\cexpect{\DW_{n+1}}\right.\\
&&\left. +\frac{\sigma^2}{4}\cexpect{(\DW_{n+1})^2}\right)\\
&=&e^{-\kappa\Dt_{n+1}}\left(X_n+\Dt_{n+1}\left(2\alpha+\frac{\sigma^2}{4}\right)\right)\\
&=&e^{-\kappa\Dt_{n+1}}(X_n+\kappa\theta\Dt_{n+1}) \\
& \leq & X_n+\kappa\theta\Dt_{n+1},
\end{eqnarray*}
where we have used the fact that $2\alpha+\sigma^2/4=\kappa\theta$ at the third step and the final inequality holds since $\kappa\geq 0$, $e^{-\kappa\Dt_{n+1}}\leq 1$.

Multiplying both sides by the indicator random variable $\mathcal{I}_{\{N^{(t)}\geq n+1\}}$, we have that a.s.
$$
\cexpect{X_{n+1}}\mathcal{I}_{\{N^{(t)}\geq n+1\}}-X_n\mathcal{I}_{\{N^{(t)}\geq n+1\}}\leq 
\kappa\theta\Dt_{n+1}\mathcal{I}_{\{N^{(t)}\geq n+1\}}.
$$
We now sum both sides over $n\in\mathbb{N}$ and take expectations to get 
\begin{multline}\label{eq:newSumNov}
\expect{\sum_{n=0}^{\infty}\left(\cexpect{X_{n+1}}\mathcal{I}_{\{N^{(t)}\geq n+1\}}-X_n\mathcal{I}_{\{N^{(t)}\geq n+1\}}\right)}\\
\leq \expect{\sum_{n=0}^{\infty}\kappa\theta\Dt_{n+1}\mathcal{I}_{\{N^{(t)}\geq n+1\}}}=\kappa\theta T.
\end{multline}
Appealing to the Dominated Convergence Theorem we can exchange the expectation and infinite sum on the LHS of \eqref{eq:newSumNov} and
\begin{eqnarray}
\lefteqn{\expect{\sum_{n=0}^{\infty}\left(\cexpect{X_{n+1}\mathcal{I}_{\{N^{(t)}\geq n+1\}}}-X_n\mathcal{I}_{\{N^{(t)}\geq n+1\}}\right)}}\nonumber\\
&=&\sum_{n=0}^{\infty}\expect{\left(X_{n+1}\mathcal{I}_{\{N^{(t)}\geq n+1\}}-X_n\mathcal{I}_{\{N^{(t)}\geq n+1\}}\right)}\nonumber\\
&=&\expect{\sum_{n=0}^{N^{(t)}-1}\left(X_{n+1}-X_n\right)}=\expect{X_{N^{(t)}}}-X_0.\label{eq:LHS}
\end{eqnarray}
Combining \eqref{eq:newSumNov} and \eqref{eq:LHS} gives the result.
\end{proof}

\subsection{An error bound in $L_1$ for CIR} 

In this section we investigate the $L_2$-strong error of the continuous form of the Lamperti based scheme \eqref{eq:CtsLamp} against the true solution of the transformed equation \eqref{eq:LampVoC}. This convergence result holds if we use a random mesh with a bounded number of steps as given in Assumption \ref{as:Nbounded} Part A, though we do not require Part B of that Assumption.

\begin{theorem}\label{lem:Yconv}
Let $(Y(t))_{t\in[0,T]}$ be a solution of \eqref{eq:Lamperti} and $\{Y_n\}_{n\in\mathbb{N}}$ be a solution of \eqref{eq:DiscLamperti}, and $\bar Y$ the continuous version given by \eqref{eq:CtsLamp}. Suppose also that Assumption \ref{as:2invMom} holds, and the (potentially random) mesh values $\{t_0,t_1,\ldots,t_N\}$ are selected  so that Assumptions \ref{as:DtMeas} and \ref{as:Nbounded} Part A hold, and 
\begin{equation}\label{eq:hmaxBound}
\max_{n}\Dt_{n} \leq \Dtmax\leq\min \left\{1,\frac{1}{\kappa}\right\}.
\end{equation}

Then there exists a constant $\Cp{3}<\infty$ such that
\begin{equation}\label{eq:errortbound}
    \max_{t\in[0,T]}\expect{\Ey{t}^2} \leq \Cp{3}\Dtmax^{1/2},
\end{equation}
where the error process 
\begin{equation}\label{eq:EyEyO}
\Ey{t}:=Y(t)-\bar Y(t),\quad t\in[0,T],
\end{equation}
and for each $n=0,
\ldots,N_{\max}$,
\begin{equation}\label{eq:L1meshpointbound}
\expect{\Ey{t_n}^2\mathcal{I}_{\{N\geq n\}}} \leq \Cp{3}\Dtmax^{1/2}.
\end{equation}
\end{theorem}

The condition $\kappa\theta>\sigma^2$ in Assumption \ref{as:2invMom} implies Assumption \ref{as:alphaPos}. Furthermore $L_1$ convergence for the  CIR scheme \eqref{eq:schemeSq} follows by an application of the Cauchy-Schwarz inequality. 

\begin{proof}
Fix $t\in[0,T]$, and let $N^{(t)}$ be as given in Definition \ref{def:Nt}. Using that
$$
e^{-\beta(u-t_n)} = e^{-\beta(u-s)}e^{-\beta(s-t_n)},\quad t_n\leq s<u\leq t_{n+1},
$$
the error process $\Ey{u}$ satisfies
\begin{multline} \label{eq:ErrorVoC}
\Ey{u}  =  e^{-\beta(u-t_n)} \Ey{t_n} \\+ \alpha\int_{t_n}^u e^{-\beta(u-s)}
\left(\frac{1}{Y(s)} - \frac{e^{-\beta(s-t_n)}}{\sqrt{(Y_n)^2 + 2\alpha(s-t_n)}} 
\right) ds \\
 + \gamma \int_{t_n}^u e^{-\beta(u-s)}\left( 1- e^{-\beta(s-t_n)} \right) dW(s),\quad u\in[t_n,t_{n+1}].
\end{multline}
We introduce the notation 
$$
\tilde{f}(s):=\frac{1}{Y(s)} - \frac{e^{-\beta(s-t_n)}}{\sqrt{(Y_n)^2+2\alpha(s-t_n)}},
$$
and re-write \eqref{eq:ErrorVoC} as the SDE on $[t_n,t_{n+1}]$
\begin{equation*}
d\Ey{u} = \left[\alpha \tilde{f}(u) - \beta \Ey{u}\right]du + \gamma\left[ 1- e^{-\beta(u-t_n)} \right] dW(u).
\end{equation*}
We can apply the It\^o formula to the vector-valued process $[\Ey{u},Y(u)]^T$ to derive, 
for $u\in[t_n,t_{n+1}]$,
\begin{multline*}
  \Ey{u}^2  =  \Ey{t_n}^2 \\+ 2\int_{t_n}^u \left(\alpha \Ey{s} \tilde{f}(s) - \beta \Ey{s}^2+\gamma^2(1-e^{-\beta(s-t_n)})^2\right) ds \\
 + 2\gamma\int_{t_n}^u \Ey{s} \left(1- e^{-\beta(s-t_n)}  \right) dW(s).
\end{multline*}
Now take expectations conditional upon $\mathcal{F}_{t_n}$,
\begin{multline}
\cexpect{\Ey{u}^2}   = \underbrace{\Ey{t_n}^2}_{=:I}+ \underbrace{2\alpha \int_{t_n}^u \cexpect{ \Ey{s}\tilde{f}(s)}ds}_{=:II}\\   +\underbrace{2\gamma^2\int_{t_n}^{u}(1-e^{-\beta(s-t_n)})ds}_{=:III} -\underbrace{ \beta \int_{t_n}^{u}\cexpect{\Ey{s}^2 } ds}_{=:IV},\quad u\in [t_n,t_{n+1}],\quad a.s.
\label{eq:condError}
\end{multline}
Note that the term $IV$ is positive since $\beta>0$, and since it is subtracted from the RHS it can be omitted in any estimate from above. For $III$, by \eqref{eq:hmaxBound} $\Dtmax<1/\kappa$, and \eqref{eq:alphabetagamma} ($\beta=\kappa/2$), so there exists $\zeta\in[0,\kappa\Dt_{n+1}]$ such that
\begin{equation}\label{eq:expTay}
1-e^{-\beta|u-t_n|}=\frac{\kappa}{2}|u-t_n|-\zeta^2/2,\quad u\in[t_n,t_{n+1}],
\end{equation}
and therefore 
\begin{equation*}
  III \leq \kappa\gamma^2(u-t_n)^2,\quad u\in[t_n,t_{n+1}].
\end{equation*}
It remains to bound $II$. There are two cases, each determined by the sign of $\Ey{s}=Y(s)-Y_n$:

\noindent \underline{Case 1:} When $Y(s)-Y_n<0$ then 
$$
e^{-\beta(s-t_n)}Y(s)-\sqrt{(Y_n)^2+2\alpha(s-t_n)}\leq Y_n-\sqrt{(Y_n)^2+2\alpha(s-t_n)}<0,
$$
and
\begin{eqnarray}
 \Ey{s}\tilde{f}(s)  &= & 
\left(Y(s)-Y_n \right)\frac{\sqrt{(Y_n)^2+2\alpha(s-t_n)}
-e^{-\beta(s-t_n)}Y(s)}{Y(s)\sqrt{(Y_n)^2 + 2\alpha(s-t_n)}}\nonumber \\
&= &\left|Y(s)-Y_n \right|\frac{e^{-\beta(s-t_n)}Y(s)-\sqrt{(Y_n)^2+2\alpha(s-t_n)}}{Y(s)\sqrt{(Y_n)^2 + 2\alpha(s-t_n)}}\label{eq:EftildeC1}
\end{eqnarray}
is negative for all $s\in[t_n,t_{n+1}]$. 

\noindent \underline{Case 2:}  Suppose $Y(s)-Y_n>0$. Then applying the inequality 
$$ 
\sqrt{(Y_n)^2 + 2\alpha(s-t_n)} \leq Y_n + \sqrt{2\alpha(s-t_n)}
$$
we get
\begin{eqnarray}
  \Ey{s}\tilde{f}(s) & = & 
\lefteqn{\left(Y(s)-Y_n \right)\frac{\sqrt{(Y_n)^2+2\alpha(s-t_n)}
-e^{-\beta(s-t_n)}Y(s)}{Y(s)\sqrt{(Y_n)^2 + 2\alpha(s-t_n)}} }\nonumber\\
&\leq&
\left(Y(s)-Y_n \right)\frac{Y_n-e^{-\beta(s-t_n)}Y(s)+\sqrt{2\alpha(s-t_n)}}
     {Y(s)\sqrt{(Y_n)^2 + 2\alpha(s-t_n)}},\nonumber
\end{eqnarray}
for all $s\in[t_n,t_{n+1}]$. On the RHS, add and subtract $Y(s)$ in the numerator and split into three terms:
\begin{eqnarray}     
\lefteqn{\Ey{s}\tilde{f}(s)}\nonumber\\
  &=&\left(Y(s)-Y_n \right)\left(\frac{Y_n-Y(s)+(1-e^{-\beta(s-t_n)})Y(s)+\sqrt{2\alpha(s-t_n)}}{Y(s)\sqrt{(Y_n)^2 + 2\alpha(s-t_n)}}\right)\nonumber \\
&=&
\frac{-(Y(s)-Y_n)^2}{Y(s)\sqrt{(Y_n)^2 + 2\alpha(s-t_n)}}+\frac{(Y(s)-Y_n)(1-e^{-\beta(s-t_n)})Y(s)}{Y(s)\sqrt{(Y_n)^2+2\alpha(s-t_n)}}\nonumber\\
&&\qquad\qquad +\frac{(Y(s)-Y_n)\sqrt{2\alpha(s-t_n)}}
{Y(s)\sqrt{(Y_n)^2 + 2\alpha(s-t_n)}},\quad s\in[t_n,t_{n+1}].\nonumber
\end{eqnarray}
For second and third terms use that $2ab \leq a^2 + b^2$ and then cancelling with first term we find
\begin{eqnarray*}     
\lefteqn{ \Ey{s}\tilde{f}(s)}\\ &\leq&
\frac{-(Y(s)-Y_n)^2}{Y(s)\sqrt{(Y_n)^2 + 2\alpha(s-t_n)}}+\frac{(Y(s)-Y_n)^2+(1-e^{-\beta(s-t_n)})^2Y(s)^2}{2Y(s)\sqrt{(Y_n)^2+2\alpha(s-t_n)}}\nonumber\\
&&\qquad\qquad +\frac{(Y(t_n)-Y_n)^2 + 2\alpha(s-t_n)}
{2Y(s)\sqrt{(Y_n)^2 + 2\alpha(s-t_n)}}\nonumber \\
&\leq&\frac{(1-e^{-\beta(s-t_n)})^2Y(s)}{2\sqrt{(Y_n)^2+2\alpha(s-t_n)}} + 
\frac{\alpha(s-t_n)}
{Y(s)\sqrt{(Y_n)^2 + 2\alpha(s-t_n)}},\quad s\in[t_n,t_{n+1}]. 
\end{eqnarray*}
By \eqref{eq:expTay} applied to the first term we have
\begin{equation}\label{eq:EftildeC2}
  \Ey{s}\tilde{f}(s) \leq \frac{4\beta^2(s-t_n)^2Y(s)}{2\sqrt{(Y_n)^2+2\alpha(s-t_n)}} + 
  \frac{\alpha(s-t_n)}{Y(s)\sqrt{(Y_n)^2 + 2\alpha(s-t_n)}},\quad s\in[t_n,t_{n+1}].
\end{equation}

Taking conditional expectation on both sides of \eqref{eq:EftildeC1} and \eqref{eq:EftildeC2} we have that the following holds for both of Cases 1 and 2 (regardless of the sign of $\Ey{s}$):
\begin{multline}     
\cexpect{\Ey{s}\tilde{f}(s)}
 \leq  \frac{2\beta^2(s-t_n)^2}{\sqrt{(Y_n)^2+2\alpha(s-t_n)}}\cexpect{Y(s)} \\ 
 + 
\frac{\alpha(s-t_n)}{\sqrt{(Y_n)^2 + 2\alpha(s-t_n)}}\cexpect{Y(s)^{-1}},\quad s\in[t_n,t_{n+1}],\quad a.s.
\label{eq:EsFt}
\end{multline}
On the RHS of \eqref{eq:EsFt} we can bound $(Y_n)^2$ by zero from below and $s-t_n\leq\Dtmax$ to get a.s. for $s\in[t_n,t_{n+1}]$,
\begin{equation*}     
\cexpect{\Ey{s}\tilde{f}(s)}
\leq \frac{\sqrt{2}\beta^2}{\sqrt{\alpha}}\Dtmax^{3/2}\cexpect{Y(s)} + 
{\sqrt{\frac{\alpha}{2}\Dtmax}}\cexpect{Y(s)^{-1}}.
\end{equation*}
Substituting back into $II$ we see that 
\begin{multline}\label{eq:IIbound1}
II=2\alpha\int_{t_n}^{u}\cexpect{\Ey{s}\tilde{f}(s)}ds
\leq 2\sqrt{2\alpha}\beta^2\Dtmax^{3/2}\int_{t_n}^{u}\cexpect{Y(s)}ds\\ + 
\alpha\sqrt{2\alpha\Dtmax}\int_{t_n}^u\cexpect{Y(s)^{-1}}ds,\quad u\in[t_n,t_{n+1}],\quad a.s.
\end{multline}
Combining all estimates for the RHS of \eqref{eq:condError} we get 
  \begin{multline}\label{eq:onestepL1}
  \cexpect{\Ey{u}^2} \leq \Ey{t_n}^2 +
 2\sqrt{2\alpha}\beta^2\Dtmax^{3/2}\int_{t_n}^{u}\cexpect{Y(s)}ds\\
+ \alpha\sqrt{2\alpha\Dtmax}\int_{t_n}^u\cexpect{Y(s)^{-1}}ds+\kappa\gamma^2(u-t_n)^2,\quad u\in[t_n,t_{n+1}],\quad a.s.
\end{multline}

On both sides of \eqref{eq:onestepL1} set $u=t_{n+1}$ and multiply both sides of \eqref{eq:onestepL1} by the indicator random variable $\mathcal{I}_{\{N^{(t)}> n+1\}}$, so that a.s.
  \begin{multline}\label{eq:onestepL1ind}
  \cexpect{\Ey{t_{n+1}}^2} \mathcal{I}_{\{N^{(t)}> n+1\}}- \Ey{t_n}^2\mathcal{I}_{\{N^{(t)}> n+1\}}\\ \leq
  2\sqrt{2\alpha}\beta^2\Dtmax^{3/2}\int_{t_n}^{t_{n+1}}\cexpect{Y(s)}\mathcal{I}_{\{N^{(t)}> n+1\}}ds\\
+ \alpha\sqrt{2\alpha\Dtmax}\int_{t_n}^{t_{n+1}}\cexpect{Y(s)^{-1}}\mathcal{I}_{\{N^{(t)}> n+1\}}ds\\+\kappa\gamma^2\Dt_{n+1}^2\mathcal{I}_{\{N^{(t)}> n+1\}}.
\end{multline}
Now sum both sides of \eqref{eq:onestepL1ind} over all steps, excluding the last step $N^{(t)}$, to get a.s.
\begin{multline}\label{eq:exceptLastStep}
    \sum_{n=0}^{N^{(t)}-2}\cexpect{\Ey{t_{n+1}}^2} \mathcal{I}_{\{N^{(t)}> n+1\}}- \sum_{n=0}^{N^{(t)}-2}\Ey{t_n}^2\mathcal{I}_{\{N^{(t)}> n+1\}}\\ \leq
  2\sqrt{2\alpha}\beta^2\Dtmax^{3/2}\sum_{n=0}^{N^{(t)}-2}\int_{t_n}^{t_{n+1}}\cexpect{Y(s)}\mathcal{I}_{\{N^{(t)}> n+1\}}ds\\
+ \alpha\sqrt{2\alpha\Dtmax}\sum_{n=0}^{N^{(t)}-2}\int_{t_n}^{t_{n+1}}\cexpect{Y(s)^{-1}}\mathcal{I}_{\{N^{(t)}> n+1\}}ds\\
+\kappa\gamma^2\sum_{n=0}^{N^{(t)}-2}\Dt_{n+1}^2\mathcal{I}_{\{N^{(t)}> n+1\}}=:J_n.
\end{multline}
Since $t\in[t_{N^{(t)}-1},t_{N^{(t)}}]$, we use \eqref{eq:onestepL1} to express the last step, noting that it holds when $t_n$ and $u$ are replaced by $t_{N^{(t)}-1}$ and $t$ respectively:
\begin{multline}\label{eq:lastStep}
      \expect{\Ey{t}^2\big|\mathcal{F}_{t_{N^{(t)}-1}}}-\Ey{t_{N^{(t)}-1}}^2 \\ \leq
 2\sqrt{2\alpha}\beta^2\Dtmax^{3/2}\int_{t_{N^{(t)}-1}}^{t}\expect{Y(s)\big|\mathcal{F}_{t_{N^{(t)}-1}}}ds\\
+ \alpha\sqrt{2\alpha\Dtmax}\int_{t_{N^{(t)}-1}}^t\expect{Y(s)^{-1}\big|\mathcal{F}_{t_{N^{(t)}-1}}}ds\\+\kappa\gamma^2(t-t_{N^{(t)}-1})^2,\quad t\in[t_{N^{(t)}-1},t_{N^{(t)}}],\quad a.s.
\end{multline}
To complete the sum to $t$, add \eqref{eq:exceptLastStep} and \eqref{eq:lastStep} and take expectations. First consider the LHS of the result. Since $N^{(t)}$ is a random integer not exceeding $N_{\max}$,  $\mathcal{I}_{\{N^{(t)}> n+1\}}$ is $\mathcal{F}_{t_n}$-measurable, and $E_Y(t_0)=0$, we get
\begin{multline}\label{eq:LHSboiled}
    \sum_{n=0}^{N_{\max}-2}\expect{E_Y(t_{n+1})^2\mathcal{I}_{\{N^{(t)}> n+1\}}-E_Y(t_{n})^2\mathcal{I}_{\{N^{(t)}> n+1\}}}\\
    +\expect{\expect{\Ey{t}^2\big|\mathcal{F}_{t_{N^{(t)}-1}}}-\Ey{t_{N^{(t)}-1}}^2}\\
    =\expect{E_Y(t_{N^{(t)}-1})^2}-\expect{E_Y(t_0)^2}+\expect{E_Y(t)^2}-\expect{E_Y(t_{N^{(t)}-1})^2}\\
    =\expect{E_Y(t)^2}.
\end{multline}
To demonstrate our approach to the RHS of the expectation of the sum of  \eqref{eq:exceptLastStep} and \eqref{eq:lastStep}, consider as an example the first sum on the RHS of \eqref{eq:exceptLastStep}. If we take an expectation we can write
\begin{eqnarray*}
\lefteqn{\expect{\sum_{n=0}^{N^{(t)}-2}\int_{t_n}^{t_{n+1}}\cexpect{Y(s)}\mathcal{I}_{\{N^{(t)}> n+1\}}ds}}\\
&=&\expect{\sum_{n=0}^{N_{\max}-2}\cexpect{\int_{t_n}^{t_{n+1}}Y(s)\mathcal{I}_{\{N^{(t)}> n+1\}}ds}}\\
&=&\sum_{n=0}^{N_{\max}-2}\expect{\cexpect{\int_{t_n}^{t_{n+1}}Y(s)\mathcal{I}_{\{N^{(t)}> n+1\}}ds}}\\
&=&\sum_{n=0}^{N_{\max}-2}\expect{\int_{t_n}^{t_{n+1}}Y(s)\mathcal{I}_{\{N^{(t)}> n+1\}}ds}\\
&=&\expect{\sum_{n=0}^{N_{\max}-2}\int_{t_n}^{t_{n+1}}Y(s)\mathcal{I}_{\{s<t_{N^{(t)}-1}\}}ds}\\
&=&\expect{\int_{0}^{t_{N^{(t)}-1}}Y(s)ds}.
\end{eqnarray*}
Now consider the entirety of the RHS of the expectation of  \eqref{eq:exceptLastStep} and follow the same steps as above:
\begin{multline}\label{eq:RHSmore}
\expect{J_n}=2\sqrt{2\alpha}\beta^2\Dt_{\max}^{3/2}\expect{\int_{0}^{t_{N^{(t)}-1}}Y(s)ds}\\+\alpha\sqrt{2\alpha\Dt_{\max}}\expect{\int_{0}^{t_{N^{(t)}-1}}Y(s)^{-1}ds}+\kappa\gamma^2t_{N^{(t)}-1}\Dtmax.
\end{multline}
Taking expectation of the RHS of \eqref{eq:lastStep} and adding to \eqref{eq:RHSmore}, along with \eqref{eq:LHSboiled}, leads to the inequality
\begin{align*}
\expect{E_Y(t)^2}
\leq& 2\sqrt{2\alpha}\beta^2\Dt_{\max}^{3/2}\expect{\int_{0}^{t}Y(s)ds}\\
&+2\alpha\sqrt{2\alpha\Dt_{\max}}\expect{\int_{0}^{t}Y(s)^{-1}ds}+\kappa\gamma^2t\Dtmax\\
\leq&\Dtmax^{1/2}t\left(2\sqrt{2\alpha}\beta^2\Dt_{\max} M_{1,1}(1+Y_0)+\alpha\sqrt{2\alpha} \Cp{2}(1,t)Y_0^{-1}+\kappa\gamma^2\Dtmax^{1/2}\right),
\end{align*}
which gives \eqref{eq:errortbound} in the statement of the theorem with 
$$
\Cp{3}=T\left(2\sqrt{2\alpha}\beta^2 M_{1,1}(1+Y_0)+\alpha\sqrt{2\alpha} \Cp{2}(1,t)Y_0^{-1}+\kappa\gamma^2\right).
$$
The same argument, but terminating the summation of \eqref{eq:onestepL1ind} at $n+1$ for any $n=0,\ldots,N_{\max}-1$, leads to the error estimate \eqref{eq:L1meshpointbound}. 
\end{proof}

\subsection{An error bound in $L_2$ for the CIR splitting scheme}
Our third main result provides an order of mean-square convergence for the scheme on a potentially random mesh, where we assume both parts of Assumption \ref{as:Nbounded}.

\begin{definition}\label{def:L2errorCns}
Let $(X(t))_{t\in[0,T]}$ be a solution of \eqref{eq:CIRVoC2} and $\{X_n\}_{n\in\mathbb{N}}$ a solution of \eqref{eq:scheme1step}. Define the error at $t_n$ to be $E_n := X(t_n) - X_n$, and define an a.s. continuous process $(\mathcal{E}^2(t))_{t\in[0,T]}$ pathwise as the a.e. linear interpolant of $E_n^2$ and $E_{n+1}^2$ on each interval $[t_n,t_{n+1}]$ for $n=0,\ldots,N-1$:
\[
\mathcal{E}^2(s):=\frac{t_{n+1}-s}{\Dt_{n+1}}E_n^2+\frac{s-t_n}{\Dt_{n+1}}E_{n+1}^2,\quad s\in[t_n,t_{n+1}],\quad a.e.
\]
where $E_Y$ is defined in \eqref{eq:EyEyO}.
 \end{definition}

\begin{theorem}\label{thm:main}
Let $(X(t))_{t\in[0,T]}$ be a solution of \eqref{eq:CIRVoC2} and $\{X_n\}_{n\in\mathbb{N}}$ a solution of \eqref{eq:scheme1step}. Suppose that Assumptions \ref{as:2invMom}, \ref{as:DtMeas} and \ref{as:Nbounded} hold, and the (potentially random) mesh values $\{t_0,t_1,\ldots,t_N\}$ are selected so that  
\begin{equation}\label{eq:hmaxBoundX}
\max_n\Dt_n\leq \Dtmax<\min\left\{1,\frac{1}{2\kappa},\frac{1}{4\kappa|1-\kappa|+\theta\kappa^2}
\right\}.
\end{equation}
Then there exists a constant $\Cp{4}<\infty$ such that
$$
\max_{t\in[0,T]}\expect{\mathcal{E}^2(t)}\leq \Cp{4}\Dtmax^{1/2},
$$
where $\mathcal{E}^2(t)$ is as in Definition \ref{def:L2errorCns}.
\end{theorem}

Since our proof relies upon Theorem \ref{lem:Yconv}, we inherit the constraint \eqref{eq:hmaxBound}, which is implied by \eqref{eq:hmaxBoundX}. 
If the mesh values are deterministic, the bound on the RHS of \eqref{eq:hmaxBoundX} becomes $\min\left\{1,\frac{1}{2\kappa}\right\}$.

\begin{proof}
Let $E_n$ be as given in Definition \ref{def:L2errorCns} and subtract the approximation \eqref{eq:scheme1step} from the variation of constants form of the true solution \eqref{eq:CIRVoC2}, evaluated at $t_n$, to get 
\begin{eqnarray}
E_{n+1} &=& e^{-\kappa\Dt_{n+1}} E_n 
    + \theta(1 - e^{-\kappa \Dt_{n+1}}) - e^{-\kappa \Dt_{n+1}} \kappa\theta \Dt_{n+1} \nonumber \\
&&    +e^{-\kappa \Dt_{n+1}} \sigma^2\left( \Dt_{n+1} -\DW_{n+1}^2\right)/4 \nonumber\\
    &&+ \sigma \int_{t_n}^{t_{n+1}} \left(e^{-\kappa(t_{n+1}-s)}\sqrt{X(s)}  - e^{-\kappa \Dt_{n+1}}
     \sqrt{X_n+2\alpha\Dt_{n+1}} \right)dW(s)\nonumber\\
     &=:& A + B + C + D,
     \label{eq:CIRABCD}
\end{eqnarray}

where
\begin{eqnarray*}
A&:=&e^{-\kappa\Dt_{n+1}}E_n;\\
B&:=&\theta(1-e^{-\kappa\Dt_{n+1}})-e^{-\kappa\Dt_{n+1}}\kappa\theta\Dt_{n+1};\\
C&:=&e^{-\kappa\Dt}\sigma^2(\Dt_{n+1}-\DW_{n+1}^2)/4;\\
D&:=& \sigma \int_{t_n}^{t_{n+1}} \left(e^{-\kappa(t_{n+1}-s)}\sqrt{X(s)}  - e^{-\kappa \Dt_{n+1}}
     \sqrt{X_n+2\alpha\Dt_{n+1}} \right)dW(s).
\end{eqnarray*}
Squaring both sides of \eqref{eq:CIRABCD} yields
$$E_{n+1}^2 = A^2 + 2AB + 2 AC + 2AD + B^2 + 2BC + 2BD+ C^2 + 2CD + D^2.$$
Taking $\mathcal{F}_{t_n}$-conditional expectations on both sides and using the bound $\Dt_{n+1}\leq \Dtmax$ in $A$ and $B$ we get that $\cexpect{AD}=\cexpect{BD}=0$.
So
\begin{multline}\label{eq:E2decomp}
\expect{E_{n+1}^2|\mathcal{F}_{t_n}} = \expect{A^2|\mathcal{F}_{t_n}} + 2\expect{AB|\mathcal{F}_{t_n}} + 2\expect{AC|\mathcal{F}_{t_n}}\\ + \expect{B^2|\mathcal{F}_{t_n}} + 2\expect{BC|\mathcal{F}_{t_n}} +
 2\expect{CD|\mathcal{F}_{t_n}}\\ + \expect{C^2|\mathcal{F}_{t_n}} + \expect{D^2|\mathcal{F}_{t_n}},\quad a.s.
\end{multline}
Now estimate \eqref{eq:E2decomp} term by term. First, write $A^2=e^{-2\kappa\Dt_{n+1}}E_n^2$. By \eqref{eq:hmaxBoundX}, $\Dtmax<1/(2\kappa)$, so there exists $\zeta\in[0,2\kappa\Dt_{n+1}]$ such that
$e^{-2\kappa\Dt_{n+1}}=1-2\kappa\Dt_{n+1}+\zeta^2/2$ 
and therefore
\begin{multline*}
\expect{A^2|\mathcal{F}_{t_n}} \leq (1-2\kappa\Dt_{n+1}+2\kappa^2\Dt_{n+1}^2)E_n^2\\=E_n^2 -2\kappa(1-\kappa\Dtmax)\Dt_{n+1}E_n^2,\quad a.s.
\end{multline*}
For $AB$ write $AB=e^{-\kappa\Dt_{n+1}}\theta(1-e^{-\kappa\Dt_{n+1}})E_n-e^{-2\kappa\Dt_{n+1}}\kappa\theta\Dt_{n+1}E_n$. Again by \eqref{eq:hmaxBoundX}, $\Dt_{n+1}<1/(2\kappa)<1/\kappa$, and therefore we can write, for some $\zeta\in[0,\kappa\Dt_{n+1}]$, 
\begin{equation*}
1-e^{-\kappa\Dt_{n+1}}(1+\kappa\Dt_{n+1})
=\kappa^2\Dt_{n+1}^2-\zeta^2/2-\kappa\Dt_{n+1}\zeta^2/2\leq \kappa^2\Dt_{n+1}^2.
\end{equation*}
Therefore 
\begin{eqnarray*}
\cexpect{AB}& = & e^{-\kappa\Dt_{n+1}}\theta\left(1-e^{-\kappa\Dt_{n+1}}(1+\kappa\Dt_{n+1}) \right)E_n\\
& \leq & e^{-\kappa\Dt_{n+1}}\theta\kappa^2\Dt_{n+1}^2|E_n|\\
& \leq & \frac{\theta}{2}\kappa^2(\Dt_{n+1}^3+\Dt_{n+1}E_n^2),\quad a.s.
\end{eqnarray*}
where we used $2ab\leq a^2 + b^2$ in the last inequality. 
For $AC$ we have
$$
\expect{AC|\mathcal{F}_{t_n}}=e^{-2\kappa\Dt_{n+1}}\frac{\sigma^2}{4}E_n\expect{\Dt_{n+1}-\DW_{n+1}^2|\mathcal{F}_{t_n}}=0,\quad a.s.
$$
For $B^2$ we have
\begin{eqnarray*}
\expect{B^2|\mathcal{F}_{t_n}}&=&\theta^2(1-e^{-\kappa\Dt_{n+1}})^2+e^{-2\kappa\Dt_{n+1}}\kappa^2\theta^2\Dt_{n+1}^2\\
&&-2\theta(1-e^{-\kappa\Dt_{n+1}})e^{-\kappa\Dt_{n+1}}\kappa\theta\Dt_{n+1}\\
&\leq&\theta^2(\kappa\Dt_{n+1}+\kappa^2\Dt_{n+1}^2/2)^2+e^{-2\kappa\Dt_{n+1}}\kappa^2\theta^2\Dt_{n+1}^2\\
&&\qquad\qquad-2\theta(\kappa\Dt_{n+1}+\kappa^2\Dt_{n+1}^2/2)e^{-\kappa\Dt_{n+1}}\kappa\theta\Dt_{n+1}\\
& = & \frac{\kappa^4\theta^2}{2}\Dt_{n+1}^4\left(\frac{9}{2} + 3 \kappa \Dt_{n+1}  + \frac{1}{2} \kappa^2 \Dt_{n+1}^2 \right), \quad a.s.
\end{eqnarray*}
For $BC$ we see that 
\begin{multline*}
BC=\frac{1}{4}\left(e^{-\kappa\Dt_{n+1}}\theta(1-e^{-\kappa\Dt_{n+1}})\sigma^2(\Dt_{n+1}-\DW_{n+1}^2)\right.\\
\left.-e^{-2\kappa\Dt_{n+1}}\kappa\theta\sigma^2(\Dt_{n+1}-\DW_{n+1}^2)\right),
\end{multline*}
from which it follows that $\expect{BC|\mathcal{F}_{t_n}}=0$ a.s.

For $C^2=\frac{1}{16}e^{-2\kappa\Dt_{n+1}}\sigma^4(\Dt_{n+1}-\DW_{n+1}^2)^2$ it follows that 
\begin{equation*}
\cexpect{C^2}=e^{-2\kappa\Dt_{n+1}}\sigma^4\Dt^2\leq\sigma^4\Dt^2,\quad a.s.
\end{equation*}

For $D^2$ we apply the It\^o isometry in its conditional form (see \cite{Mao}) to get a.s,
\begin{multline}\label{eq:dot1}
\expect{D^2|\mathcal{F}_{t_n}}\\
=\sigma^2\int_{t_n}^{t_{n+1}}\cexpect{\left(e^{-\kappa(t_{n+1}-s)}\sqrt{X(s)}-e^{-\kappa\Dt_{n+1}}\sqrt{X_n+2\alpha\Dt_{n+1}}\right)^2}ds.
\end{multline}
To the integrand on the RHS of \eqref{eq:dot1} we apply the following a.s. bound
\begin{eqnarray}\label{eq:Ctilde}
\lefteqn{\cexpect{\left( e^{-\kappa(t_{n+1}-s)}\sqrt{X(s)}-e^{-\kappa\Dt_{n+1}}\sqrt{X_n+2\alpha\Dt_{n+1}}\right)^2} } \nonumber\\ 
& \leq &
2 (e^{-\kappa(t_{n+1}-s)}-e^{-\kappa\Dt_{n+1}})^2 \cexpect{X(s)} \nonumber \\
 & & + 
 2e^{-2\kappa\Dt_{n+1}}\cexpect{\left(\sqrt{X(s)}-\sqrt{X(t_n)}\right)^2}\nonumber \\
& & + 
2e^{-2\kappa\Dt_{n+1}}\left(\sqrt{X(t_n)}-\sqrt{X_n}\right)^2\nonumber \\
& & + 
2e^{-2\kappa\Dt_{n+1}}\left(\sqrt{X_n}-\sqrt{X_n+2\alpha\Dt_{n+1}}\right)^2 \nonumber \\
& =: & \tilde{D}_{1,n} + 
\tilde{D}_{2,n} +
\tilde{D}_{3,n} + 
\tilde{D}_{4,n}.\nonumber
\end{eqnarray}
Bounding each in turn, we first apply \eqref{eq:MeanRevert} in the statement of Lemma \ref{lem:MeanRevert} to get:
\[
\tilde{D}_{1,n}\leq 4\kappa\Dt_{n+1}^2(X(t_n)+\theta),\quad a.s.
\]
Second, apply \eqref{eq:Yreg} in the statement of Lemma \ref{lem:adap} to get
\[
\tilde{D}_{2,n}\leq 8\gamma^2\Dt_{n+1}+2\bar L_{n}\Dt_{n+1}^2,\quad a.s.
\]
Third, note that $\tilde{D}_{3,n}\leq 2\Ey{t_n}^2$, where $E_Y$ is defined by \eqref{eq:EyEyO} in the statement of Theorem \ref{lem:Yconv}. Fourth, we multiply out the square in $\tilde{D}_{4,n}$ to get the bound
\[
\tilde{D}_{4,n}\leq 4\alpha\Dt_{n+1}.
\]

Bringing $\tilde{D}_{1,n},\tilde{D}_{2,n},\tilde {D}_{3,n}, \tilde{D}_{4,n}$ together we get
\begin{eqnarray}
\cexpect{D^2}&\leq&\sigma^2\Dt_{n+1}^{3/2}\left[4\kappa\Dt_{n+1}^{3/2}(X(t_n)+\theta)+8\gamma^2\Dt_{n+1}^{1/2}\right.\nonumber\\
&&\left.\qquad\qquad+2\bar L_{n}\Dt_{n+1}^{3/2}+4\alpha\Dt_{n+1}^{1/2}\right]+2\sigma\Ey{t_n}^2\Dt_{n+1}\nonumber\\
&\leq& \bar K_{5,n}\Dt_{n+1}^{3/2}+2\sigma\Ey{t_n}^2\Dt_{n+1},\quad a.s,\label{eq:D2bound} 
\end{eqnarray}
where using the fact that \eqref{eq:hmaxBoundX} ensures $\Dt_{n+1}\leq 1$, we define the random variable $\bar K_{5,n}:=\sigma^2[4\kappa(X(t_n)+\theta)+8\gamma^2+2\bar L_{n}+4\alpha]$. For $CD$ we use the bound from $D^2$ to get a.s,
\begin{multline*}
\expect{CD|\mathcal{F}_{t_n}}=\frac{1}{4}e^{-\kappa\Dt_{n+1}}\sigma^3\mathbb{E}\left[(\Dt_{n+1}-\DW_{n+1}^2)\right.\\\left.
\times\int_{t_n}^{t_{n+1}}\left(e^{-\kappa(t_{n+1}-s)}\sqrt{X(s)}-e^{-\kappa\Dt_{n+1}}\sqrt{X_n+2\alpha\Dt_{n+1}}\right)dW(s)|\mathcal{F}_{t_n}\right].
\end{multline*}
Then applying the Cauchy-Schwarz inequality we get
\begin{eqnarray*}
\expect{CD|\mathcal{F}_{t_n}}
& \leq& \sqrt{\expect{(\Dt_{n+1}-\DW_{n+1}^2)^2|\mathcal{F}_{t_n}}}\sqrt{\expect{D^2|\mathcal{F}_{t_n}}}
\nonumber \\
&=&\Dt_{n+1}\sqrt{2\expect{D^2|\mathcal{F}_{t_n}}}\\
&\leq& \Dt_{n+1}^{3/2}\sqrt{2\bar{K}_{5,n}\Dt_{n+1}^{1/2}+4\sigma\Ey{t_n}^2}\\
&\leq& \sqrt{2\bar{K}_{5,n}}\Dt_{n+1}^{7/4}+\sqrt{4\sigma \Ey{t_n}^2}\Dt_{n+1}^{3/2},\quad a.s,
\end{eqnarray*}
making use of the bound on $\cexpect{D^2}$ given by \eqref{eq:D2bound}. By the standard inequality $2ab\leq a^2 + b^2$ we find 
$$
\expect{CD|\mathcal{F}_{t_n}}\leq 
\sqrt{2\bar{K}_{5,n}}\Dt_{n+1}^{7/4}+2\sigma \Ey{t_n}^2\Dt_{n+1} + \frac{1}{2}\Dt_{n+1}^{2},\quad a.s.
$$

Substituting estimates of $A^2, AB, AC, B^2, BC, CD, C^2$ and $D^2$ into \eqref{eq:E2decomp} gives
\begin{multline*}
\cexpect{E_{n+1}^2}\leq E_n^2-2\kappa(1-\kappa\Dtmax)\Dt_{n+1}E_n^2+\Dt_{n+1}E_{n}^2\frac{\theta\kappa^2}{2}\\
+\Dt_{n+1}^3\frac{\theta\kappa^2}{2}+\frac{\kappa^4\theta^2}{2}\Dt_{n+1}^4\left(\frac{9}{2} + 3 \kappa \Dt_{n+1}  + \frac{1}{2} \kappa^2 \Dt_{n+1}^2 \right)\\+2\sqrt{2\bar K_{5,n}}\Dt_{n+1}^{7/4}+\left(\sigma^4+\frac{1}{2}\right)\Dt_{n+1}^2+\bar K_{5,n}\Dt_{n+1}^{3/2}+4\sigma\Ey{t_n}^2\Dt_{n+1},\quad a.s.
\end{multline*}
Rearranging terms and again using that, by \eqref{eq:hmaxBoundX} $\Dtmax< 1$, we get
\begin{eqnarray}
\cexpect{E_{n+1}^2}
& \leq &  \left(1-2\kappa(1-\kappa\Dtmax)\Dt_{n+1}+\Dt_{n+1}\frac{\theta\kappa^2}{2}\right)E_n^2 \nonumber \\
&&+\Dt_{n+1}^3\frac{\theta\kappa^2}{2}+\frac{\kappa^4\theta^2}{2}\Dt_{n+1}^4\left(\frac{9}{2} + 3 \kappa \Dt_{n+1}  + \frac{1}{2} \kappa^2 \Dt_{n+1}^2 \right)\nonumber \\
&&+2\sqrt{2\bar K_{5,n}}\Dt_{n+1}^{7/4}+\left(\sigma^4+\frac{1}{2}\right)\Dt_{n+1}^2\nonumber\\
&&+\bar K_{5,n}\Dt_{n+1}^{3/2}+4\sigma\Ey{t_n}^2\Dt_{n+1}\nonumber \\
&\leq&\left(1-\Dt_{n+1}\left(2\kappa(1-\kappa)+\frac{\theta\kappa^2}{2}\right)\right)E_n^2\nonumber
\nonumber \\
&&+\Dt_{n+1}^2\left(\frac{\theta\kappa^2}{2}+\frac{\kappa^4\theta^2}{2}\left(\frac{9}{2} + 3 \kappa  + \frac{1}{2} \kappa^2\right) + \sigma^4+\frac{3}{2} \right)\nonumber\\
&&+3\bar K_{5,n}\Dt_{n+1}^{3/2}
+4\sigma\Ey{t_n}^2\Dt_{n+1},\quad a.s.
\label{eq:LampertiPreGronwall}
\end{eqnarray}

Rearranging \eqref{eq:LampertiPreGronwall} and multiplying both sides by the indicator random variable $\mathcal{I}_{\{N^{(t)}\geq n+1\}}$ gives a.s.
\begin{multline}\label{eq:Endropped}
\cexpect{E_{n+1}^2\mathcal{I}_{\{N^{(t)}\geq n+1\}}}-E_n^2\mathcal{I}_{\{N^{(t)}\geq n+1\}}\\\leq
\Dt_{n+1}\left(2\kappa|1-\kappa|+\frac{\theta\kappa^2}{2}\right)E_n^2\mathcal{I}_{\{N^{(t)}\geq n+1\}}\\+
\Dt_{n+1}^2\left(\frac{\theta\kappa^2}{2}+\frac{\kappa^4\theta^2}{2}\left(\frac{9}{2} + 3 \kappa  + \frac{1}{2} \kappa^2\right) + \sigma^4+\frac{3}{2} \right)\mathcal{I}_{\{N^{(t)}\geq n+1\}}\\+3\bar K_{5,n}\Dt_{n+1}^{3/2}\mathcal{I}_{\{N^{(t)}\geq n+1\}}+4\sigma\Ey{t_n}^2\Dt_{n+1}\mathcal{I}_{\{N^{(t)}\geq n+1\}} =:V_n.
\end{multline}
It follows from Definition \ref{def:L2errorCns} that
$(t_{n+1}-s)E_n^2\leq\Dt_{n+1}\mathcal{E}^2(s)$, a.s. for $s\in[t_n,t_{n+1}]$
and therefore by integration
\begin{equation}\label{eq:EnboundL2}
\Dt_{n+1}E_n^2\leq 2\int_{t_n}^{t_{n+1}}\mathcal{E}^2(s)ds,\quad a.e.    
\end{equation}
The a.s. continuity of $(\mathcal{E}^2(s))_{s\in [0,T]}$ implies the continuity and therefore boundedness over $[0,T]$ of $\expect{\mathcal{E}^2(t)}$. 

Summing the LHS of \eqref{eq:Endropped} over all steps, and taking expectations yields, 
\begin{eqnarray}
\lefteqn{\expect{\sum_{n=0}^{N^{(t)}-1}\left(\cexpect{E_{n+1}^2\mathcal{I}_{\{N^{(t)}\geq n+1\}}}-E_n^2\mathcal{I}_{\{N^{(t)}\geq n+1\}}\right)}}\nonumber\\
&=&\sum_{n=0}^{N_{\max}-1}\expect{\left(E_{n+1}^2\mathcal{I}_{\{N^{(t)}\geq n+1\}}-E_n^2\mathcal{I}_{\{N^{(t)}\geq n+1\}}\right)}\nonumber\\
&=&\expect{\sum_{n=0}^{N^{(t)}-1}\left(E_{n+1}^2-E_n^2\right)}=\expect{E_{N^{(t)}}^2},\label{eq:LHSE}
\end{eqnarray}
where we have used Assumption \ref{as:Nbounded} Part A, and that $E_0^2=0$ at the final step.

Summing the RHS of \eqref{eq:Endropped} over all steps, taking expectations, and applying \eqref{eq:EnboundL2} yields the finite estimate
\begin{align}
\expect{\sum_{n=0}^{N^{(t)}-1} V_n}
&\leq  \left(4\kappa|1-\kappa|+\theta\kappa^2\right)\expect{\int_{0}^{t_{N^{(t)}}}\mathcal{E}^2(s)ds}\nonumber \\
&\quad+T\left(\frac{\theta\kappa^2}{2}+\frac{\kappa^4\theta^2}{2}\left(\frac{9}{2} + 3 \kappa  + \frac{1}{2} \kappa^2\right) + \sigma^4+\frac{3}{2} \right)\Dtmax\nonumber\\
&\quad+3TK_5\Dtmax^{1/2}+4\sigma N_{\max}\Dtmax^{3/2}\Cp{3},\label{eq:RHSE}
\end{align}
where $K_5=\expect{\bar K_{5,n}}<\infty$ and we have used \eqref{eq:L1meshpointbound} in the statement of Theorem \ref{lem:Yconv} at the final step, and the fact that $\mathcal{I}_{\{N^{(t)}\geq n+1\}}\leq \mathcal{I}_{\{N^{(t)}\geq n\}}$ a.s. 

By \eqref{eq:TmaxDef} in Assumption \ref{as:Nbounded} Part B, and substituting \eqref{eq:LHSE}, \eqref{eq:RHSE} into \eqref{eq:Endropped} we get
\begin{multline*}
\expect{E_{N^{(t)}}^2}\leq \left(4\kappa|1-\kappa|+\theta\kappa^2\right)\left(\expect{\int_{0}^{t}\mathcal{E}^2(s)ds}+\expect{\int_{t}^{t_{N^{(t)}}}\mathcal{E}^2(s)ds}\right)\\
+T\left(\frac{\theta\kappa^2}{2}+\frac{\kappa^4\theta^2}{2}\left(\frac{9}{2} + 3 \kappa  + \frac{1}{2} \kappa^2\right) + \sigma^4+\frac{3}{2} \right)\Dtmax\\
+3TK_5\Dtmax^{1/2}+4\sigma T_{\max}\Cp{3}\Dtmax^{1/2}.
\end{multline*}
A similar argument, along with the a.s. positivity of $\mathcal{E}^2(s)$, gives
\begin{multline*}
\expect{E_{N^{(t)}-1}^2}\leq \left(4\kappa|1-\kappa|+\theta\kappa^2\right)\expect{\int_{0}^{t}\mathcal{E}^2(s)ds}\\
+T\left(\frac{\theta\kappa^2}{2}+\frac{\kappa^4\theta^2}{2}\left(\frac{9}{2} + 3 \kappa  + \frac{1}{2} \kappa^2\right) + \sigma^4+\frac{3}{2} \right)\Dtmax\\
+3TK_5\Dtmax^{1/2}+4\sigma T_{\max}\Cp{3}\Dtmax^{1/2}.
\end{multline*}
Definition \ref{def:L2errorCns} implies that 
$$\mathcal{E}^2(s)\leq\max\left\{E_{N^{(t)}-1}^2,E_{N^{(t)}}^2\right\} \leq E_{N^{(t)}-1}^2+E_{N^{(t)}}^2
$$ 
for all $s\in[t_{N^{(t)}-1},t_{N^{(t)}}]$ a.e.  This means we can write
\begin{multline*}
\expect{E_{N^{(t)}}^2}+\expect{E_{N^{(t)}-1}^2}\leq
2K_6\expect{\int_{0}^{t}\mathcal{E}^2(s)ds}
+K_7\Dtmax+3TK_5\Dtmax^{1/2}\\
+K_6\Dtmax\left(\expect{E_{N^{(t)}}^2}+\expect{E_{N^{(t)}-1}^2}\right)+4\sigma T_{\max}\Cp{3}\Dtmax^{1/2},
\end{multline*}
where $K_6:=4\kappa|1-\kappa|+\theta\kappa^2$ and $K_7:=T\left(\theta\kappa^2+\kappa^4\theta^2\left(9 + 6 \kappa  +  \kappa^2\right) + 2\sigma^4+3\right)$. Rearranging, we get
\begin{multline*}
    \expect{E_{N^{(t)}}^2}+\expect{E_{N^{(t)}-1}^2}
    \leq
\frac{1}{1-K_6\Dtmax}\left(2K_6\expect{\int_{0}^{t}\mathcal{E}^2(s)ds}\right.\\
\left.
+K_7\Dtmax+(3TK_5+4\sigma T_{\max}\Cp{3})\Dtmax^{1/2}\right).
\end{multline*}
By \eqref{eq:hmaxBoundX}, $\Dtmax<1/K_6=1/(4\kappa|1-\kappa|+\theta\kappa^2)$, and setting $K_8:=K_7\vee (3TK_5+4\sigma T_{\max}\Cp{3})$ we write
\begin{eqnarray*}
\expect{\mathcal{E}^2(t)}
    &\leq&
\frac{1}{1-K_6\Dtmax}\left(2K_6\expect{\int_{0}^{t}\mathcal{E}^2(s)ds}\right.\\
&&\left.+K_7\Dtmax+(3TK_5+4\sigma T_{\max}\Cp{3})\Dtmax^{1/2}\right)\\
&\leq& 2K_6\int_{0}^{t}\expect{\mathcal{E}^2(s)}ds+K_8\Dtmax^{1/2}.
\end{eqnarray*}
An application of Gronwall's inequality gives
\[
\expect{\mathcal{E}^2(t)}\leq K_8e^{2TK_6}\Dtmax^{1/2},
\]
from which the statement of the theorem follows with $\Cp{4}=K_8e^{2TK_6}$.
\end{proof}

\subsection{A local error bound for the scheme in the \emph{soft zero} region}\label{sec:LEBSZ}
Finally, we show that one step in the {\em soft zero region} preserves the local error from \eqref{eq:scheme1step}, that is the {\em soft zero} does not change the local error rate (and hence should not change the global error rate). 

\begin{lemma}\label{lem:softzeroLocalError}
Let $(X(t))_{t\in[0,T]}$ be a solution of \eqref{eq:CIR}, $X_n < \Xzero$, $X_{n+1}=\Xzero$ and $\Dt_{n+1}$ be given by \eqref{eq:dtsoftzero}. Suppose that either of Assumptions \ref{as:alphaPos} or \ref{as:alphaNeg} holds, and suppose also that for $r,s>0$
\begin{equation}\label{eq:expectassum}
\expect{|X(t_n)-X_n|}\leq \tilde C_n\Dtmax^r
\quad \text{and} \quad \expect{(X(t_n)-X_n)^2} \leq \tilde C_n\Dtmax^s.
\end{equation}
Then there exist constants $\tilde C_n,\tilde C_{n+1}>0$, independent of $\Dtmax$, such that 
$$\expect{(X(t_{n+1})-\Xzero)^2}\leq \tilde C_{n+1}\Dtmax^{\min(s,r+1)}.$$ 
\end{lemma}
\begin{proof}
By construction,
$$X_{n+1}=\Xzero=X_n + \kappa \int_{t_n}^{t_{n+1}} (\theta - X_n) ds.$$
Define $\Eonestep_{n+1}:=X(t_{n+1})-\Xzero$ and 
$\Eonestep_n:=X(t_n)-X_n$ 
then 
$$
\Eonestep_{n+1} = \Eonestep_n - \kappa \int_{t_n}^{t_{n+1}} (X(s)-X_n) ds + \gamma\int_{t_n}^{t_{n+1}} \sqrt{X(s)} dW(s).
$$
Squaring both sides, using the elementary inequality $(a+b+c)^2 \leq 3a^2 + 3b^2+3c^2$ and a standard corollary of Jenzen's inequality,
 \begin{equation*}
\Eonestep_{n+1}^2 \leq   3\Eonestep_n^2 
 + 3\kappa^2\Dt_{n+1}\int_{t_n}^{t_{n+1}} (X(s)-X_n)^2 ds 
+ 3\gamma^2 \left(\int_{t_n}^{t_{n+1}} \sqrt{X(s)} dW(s)\right)^2.
\end{equation*}
Adding in and subtracting out $X(t_n)$ in the first integral, then taking conditional expectation and applying the It\^o Isometry, again in its conditional form, we get
 \begin{multline*}
\cexpect{\Eonestep_{n+1}^2} \leq   3\Eonestep_n^2(1+2\kappa^2\Dt_{n+1}^2)
 \\+ 6\kappa^2\Dt_{n+1}\int_{t_n}^{t_{n+1}} \cexpect{(X(s)-X(t_n))^2} ds  
\\ + 3\gamma^2 \int_{t_n}^{t_{n+1}} \cexpect{X(s)} ds,\quad a.s.
\end{multline*}
Now by the bound \eqref{eq:PREMeanRevert} in the statement of Lemma \ref{lem:MeanRevert} we find
 \begin{multline*}
\cexpect{\Eonestep_{n+1}} \leq   3\Eonestep_n^2(1+2\kappa^2\Dt_{n+1}^2)\\
 + 6\kappa^2\Dt_{n+1}\int_{t_n}^{t_{n+1}} \cexpect{(X(s)-X(t_n))^2} ds  
\\ + 3\gamma^2 \Dt_{n+1}(X(t_n)+\theta\Dt_{n+1}),\quad a.s.
\end{multline*}
For the last term we add in and subtract out $X_n$ and use that $X_n\leq \Xzero$ 
\begin{multline*}
\cexpect{\Eonestep_{n+1}^2}\\ \leq   3\Eonestep_n^2(1+2\kappa^2\Dt_{n+1}^2) 
 + 6\kappa^2\Dt_{n+1}\int_{t_n}^{t_{n+1}} \cexpect{(X(s)-X(t_n))^2} ds  \\
 + 3\gamma^2 \Dt_{n+1}\Eonestep_n+ 3\gamma^2( \Xzero+\theta\Dt_{n+1})\Dt_{n+1},\quad a.s.
\end{multline*}
For the second term it suffices to use the bound \eqref{eq:equCondMp1} in Lemma \ref{lem:posMB} to get 
\begin{multline*}
\cexpect{\Eonestep_{n+1}^2} \leq   3\Eonestep_n^2(1+2\kappa^2\Dt_{n+1}^2)
 + 24\kappa^2\Dt_{n+1}^2M_{1,2}(1+X(t_n)^2)  \\
 + 3\gamma^2 \Dt_{n+1}\Eonestep_n+ 3\gamma^2( \Xzero+\theta\Dt_{n+1})\Dt_{n+1},\quad a.s.
\end{multline*}
Since $\Xzero\leq \rho^{-1}\theta\Dtmax$ we have, after taking the expectation and using \eqref{eq:BossyDiopMomentBound} with $p=1$, 
\begin{multline*}
\expect{\Eonestep_{n+1}^2} \leq   3\expect{\Eonestep_n^2}(1+2\kappa^2\Dtmax^2)
 + 24\kappa^2(M_{1,2}(1+M_2))\Dtmax^2 \\ 
 + 3\gamma^2 \expect{\Eonestep_n}\Dtmax+ 3\gamma^2\theta( \rho^{-1}+1)\Dtmax^2.
\end{multline*}
An application of the bounds \eqref{eq:expectassum} yields the result.
\end{proof}

\section{Numerical results}
\label{sec:numres}
We compare our new splitting/adaptive method to four other methods in the literature. We let $\Dt$ denote a fixed time step and $\Delta W_{n+1}$ be an increment of the Brownian motion.
As a reference solution we take 
the Milstein method of \cite{HefterHerzwurm2018}, which is known to converge strongly over all $\alpha>0$. For \eqref{eq:CIR} this is 
\begin{align}
     R_1 := & \max\left\{\sigma\sqrt{\Dt}/2,\sqrt{\max\{\sigma^2\Dt/4,X_n\}}+\frac{\sigma}{2}\DW_{n+1}\right\};\nonumber \\
     X_{n+1} = & \max\left\{R_1^2+ \Dt(\kappa(\theta-X_n)-\sigma^2/4),0  \right\}.\nonumber 
\end{align}
In our figures below we denote this method as {\tt Milstein}. 
We also consider the fully truncated method
(denoted {\tt Fully Truncated}) proposed by \cite{LordKoekkeokvanDijk2010} given by 
\begin{eqnarray*}
\widetilde{X}_{n+1}&=&\widetilde{X}_n+\Dt\kappa(\theta-\max\{\widetilde{X}_n,0\})+\sigma\sqrt{\max\{\widetilde{X}_n,0\}}\Delta W_{n+1};\\
X_{n+1}&=&\max\{ \widetilde{X}_{n+1}, 0 \}.
\end{eqnarray*}
Both these approaches maintain non-negativity of the numerical approximation by enforcing that the approximated solution is always greater or equal to zero. For larger $\Dt$ values this may then be enforced over a few consecutive steps. This is in contrast to 
the {\em soft zero} approach that we take for $\alpha<0$, ($\sigma>2\sqrt{\kappa\theta}$) that, we believe, better mimics the dynamics of the underlying SDE \eqref{eq:CIR}, (see below). 

The last two methods we compare to are
 derived from the Lamperti transformation.  First is the drift implicit method (denoted {\tt Implicit}) of \cite{Alfonsi2013} 
\begin{equation*}
    Y_{n+1}=\frac{Y_n+\gamma\Delta W_{n+1}}{2(1-\beta \Dt)}+\sqrt{\frac{(Y_n+\gamma\Delta W_{n+1})^2}{4(1-\beta \Dt)^2}+\frac{\alpha \Dt}{1-\beta \Dt}},
\end{equation*}
where  $\alpha, \beta$ and $\gamma$ are defined in \eqref{eq:alphabetagamma}. Recall that $X(t_n)\approx (Y_n)^2$. This scheme has not been extended to the $\alpha\leq 0$ regime $(\sigma\geq 2\sqrt{\kappa\theta})$ and so we compute it only for $\alpha>0$. The second is the projection based method of \cite{CJM16} (denoted {\tt Projected}), which for CIR reduces to 
\begin{equation*}
    Y_{n+1} = \hat{Y_n} + \left(\frac{\alpha}{\hat{Y_n}}-\beta\hat{Y_n}\right)\Dt_{n+1} + \gamma \Delta W_{n+1},
\quad \hat{Y_n}:=\max(N^{-0.25},Y_{n}).
\end{equation*}

Although not covered by our convergence analysis, we fix the initial data $X_0=0$ in order to ensure the numerical methods must operate with values in the neighbourhood of zero. We fix parameters $\kappa=2$, $\theta=0.02$ and vary the parameter $\sigma$. In Table \ref{tab:sigmavals} we summarize some key values of $\sigma$ as the relate to the theory and as shown in the figures.
\begin{table}
\small{
\begin{center}
\begin{tabular}{ |c|c|c| } 
 \hline
 Description & $\sigma$ & Value  \\ 
 \hline \hline
 \makecell{
Projected Euler reaches rate $1/4$\\
Limit of theory for truncated Euler (rate $1/2$)}
& $\sigma=\sqrt{2\kappa\theta/3}$ & $\approx 0.1633$ \\\hline
 \makecell{
Limit of theory for Splitting (rate $1/4$, $p=1,2$)\\
Limit of theory for Drift Implicit (rate $1$, $p<2$) \\
Limit of theory for truncated Euler (rate $1/6$, $p=1$)}
& $\sigma=\sqrt{\kappa\theta}$ & $0.2$ \\  \hline 
Feller condition & $\sigma=\sqrt{2\kappa\theta}$ & $\approx 0.2828$ \\ \hline
 \makecell{
$\alpha\leq 0$. 
Adaptivity and Soft-Zero\\ required for splitting method}
& $\sigma\geq 2\sqrt{\kappa\theta}$ & $0.4$ \\ \hline
\end{tabular}
\end{center}
}
\caption{Key $\sigma$ values when $\kappa=2$, $\theta=0.02$. See Table \ref{tab:comparetheory} on comparisons of theoretical results.}
\label{tab:sigmavals} 
\end{table}

We base our comparisons on $M=1000$ realizations and include the uncertainty in estimated quantities based on $20$ batches of $50$ samples.
Reference solutions are computed with $\Dtref=10^{-5}$ and $\Xzero=\theta(1-e^{-\kappa\Dtref}) / \rho$ with $\rho=2$, and 
$$\Dtmax\in\{0.1,0.01,0.005,0.001, 0.0005, 0.0001\}. $$
To compute rates of convergence we fit a linear polynomial to the data.

In Fig. \ref{fig:samplepaths} we compare for $\sigma=0.8$ sample paths from the four schemes proposed for this regime (so do not include {\tt Implicit} as $\alpha<0$).
 The left column $(a)$ and $(c)$ has $\Dtmax=10^{-5}$ for our splitting/adaptive method. In (a) the fixed step methods were all taken with $\Dt=\Dtmax$. 
 In the right column $(b)$ and $(d)$ has $\Dtmax=0.01$ for our splitting/adaptive method; the fixed step methods
 $$\overline{\Dt}=\frac1N\sum_{n=0}^{N-1} \Dt_{n+1} \approx 0.004$$
 from the splitting/adaptive method. 
 In the timestep plots $(c)$ and $(d)$ we indicate with a circle which steps were taken using the deterministic step. We also show in the upper horizontal line where $\Dt_{n+1}=0.95\frac{\Xzero}{2|\alpha|}$ (from  \eqref{eq:adaptdt}) and also where $\Dt_{n+1}=\Xzero$ (lower horizontal line).

The error constant for the projected method of \cite{CJM16} is consistently largest for all realizations and $\sigma$ values, this is evident from the paths in both (a) and (b).
 In (a) our method may take smaller steps than the fixed step methods, whereas in (b) using the average stepsize for the fixed step methods means we have a larger error where, for example, $\Dt_n=\Dtmax$ (e.g. for $t\in[0.18,0.425]$).
 
Comparing the left columns ($(a)$ and $(c)$) we see that the adaptivity occurs where the method is close to zero and where the solution is large the maximum time step can be used. Small steps may arise (e.g. $10^{-9}$) where the solution just enters the {\em soft zero} region.
We observe in $(b)$ (e.g. around $t=0.75$) that the fully truncated scheme is liable to miss dynamics in the region of zero,
though this does not seem to be an issue for the Milstein method in this realization.
\begin{figure}
\centering
\begin{subfigure}[t]{0.49\textwidth}
\includegraphics[width=\textwidth]{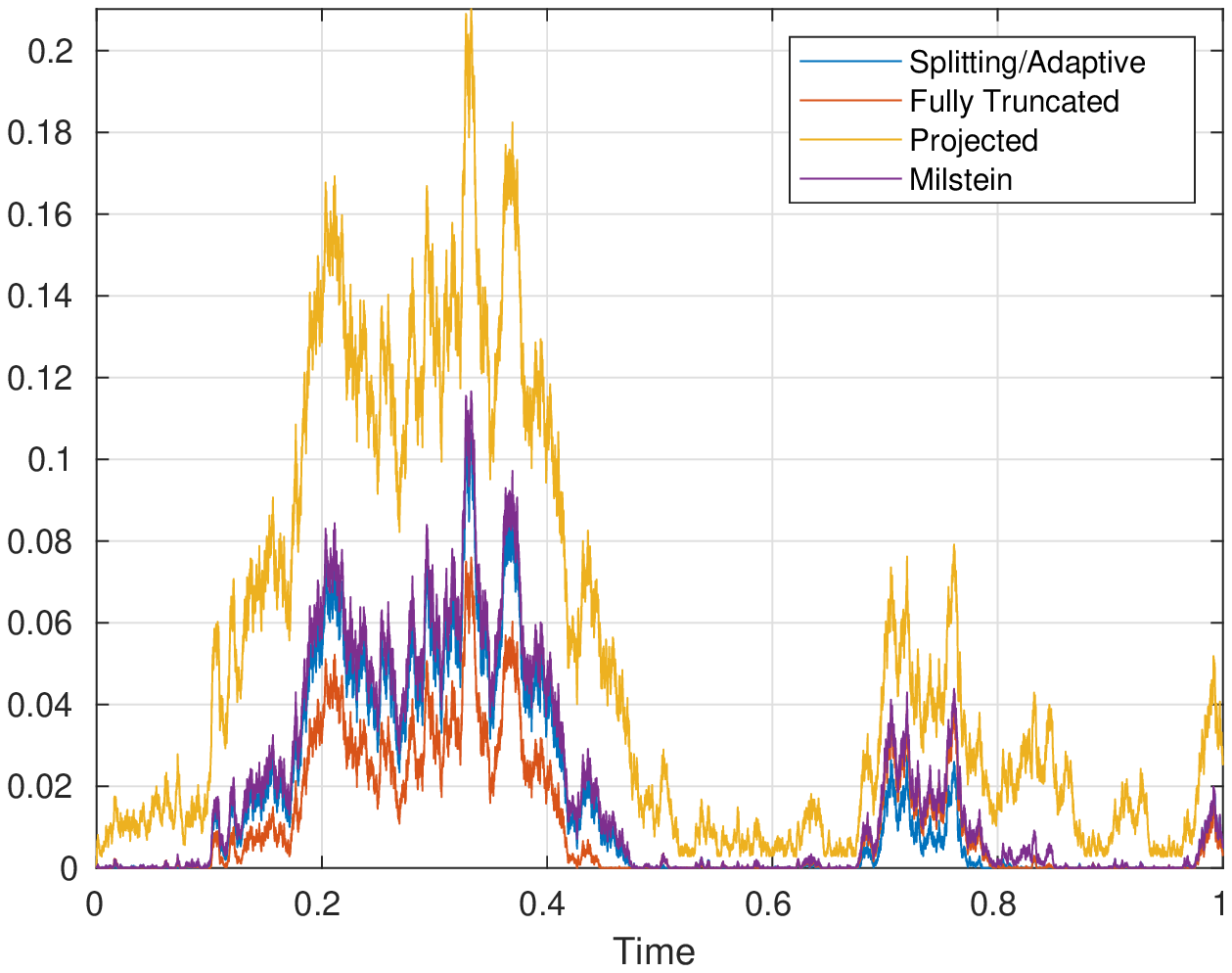}
\caption{$\Dtmax=10^{-5}$}
\end{subfigure}
\begin{subfigure}[t]{0.49\textwidth}
\includegraphics[width=\textwidth]{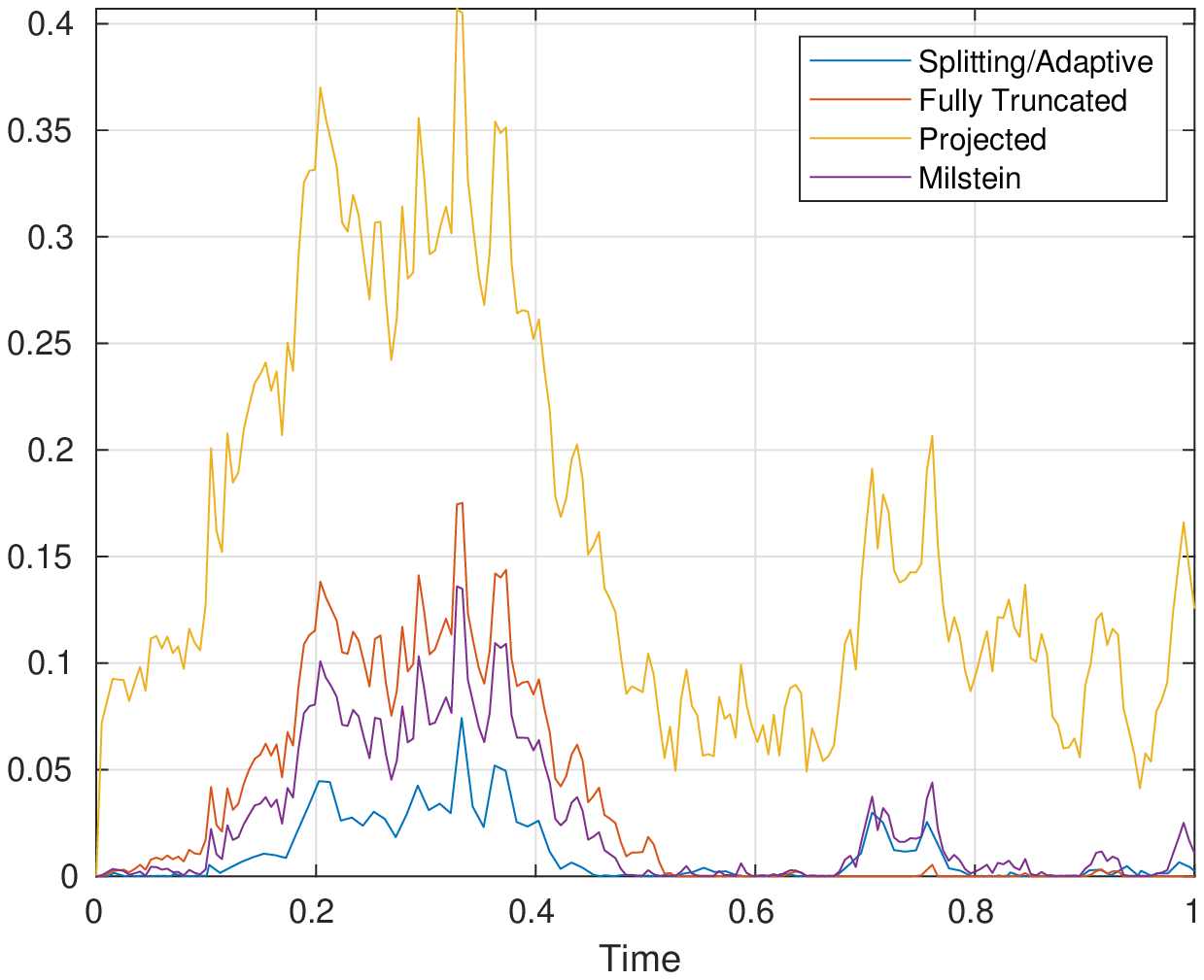}
\caption{$\Dtmax=0.01$}
\end{subfigure}
\begin{subfigure}[b]{0.49\textwidth}
\includegraphics[width=\textwidth]{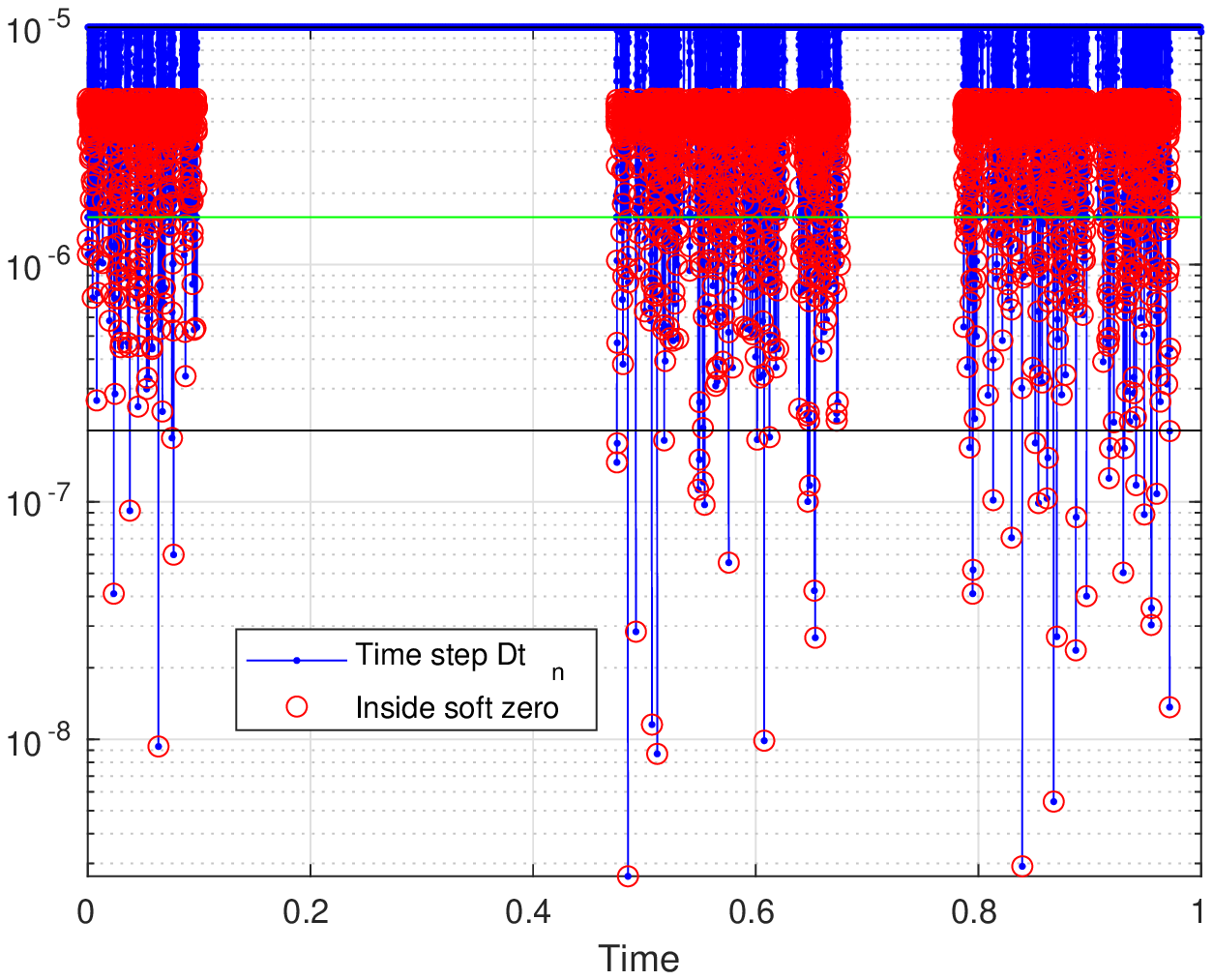}
\caption{$\Dtmax=10^{-5}$}
\end{subfigure}
\begin{subfigure}[b]{0.49\textwidth}
\includegraphics[width=\textwidth]{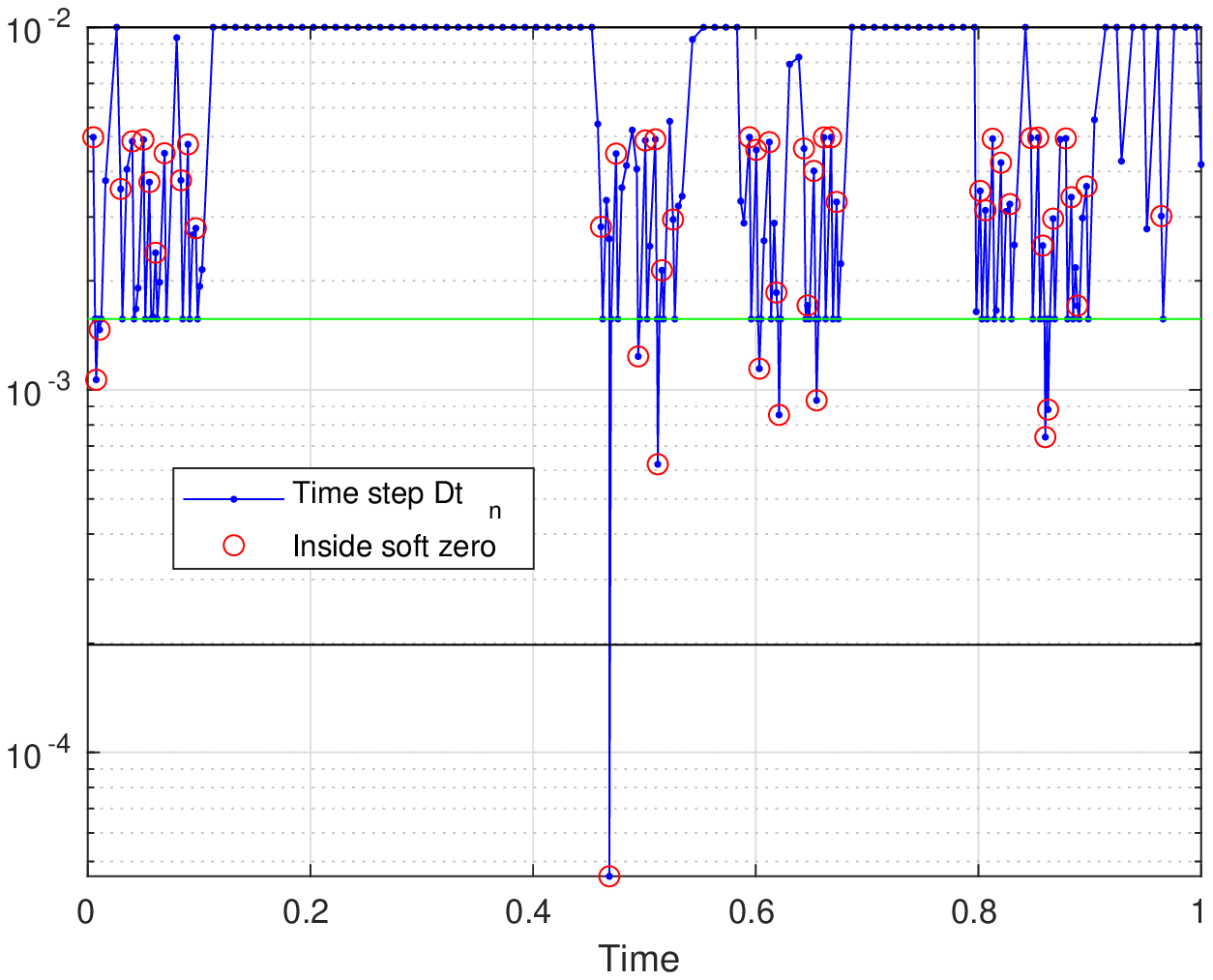}
\caption{$\Dtmax=0.01$}
\end{subfigure}
\caption{Sample paths and adaptive steps taken for $\sigma=0.8$ with $\Dtmax=\Dtref=10^{-5}$ in (a), (c) and $\Dtmax=0.01$ in (b) and (d). On the time steps in (c) and (d) the circles represent steps taken with the deterministic ODE.
We also indicate for reference where $\Dt=\Xzero$ (lower horizontal line) and where \eqref{eq:adaptdt}
holds with $\Dt_{n+1}=\Xzero$.}
\label{fig:samplepaths}
\end{figure}

In Fig.~\ref{fig:rateM1000} we plot the rate of convergence for all the schemes 
as $\sigma$ increases based on $M=1000$ samples. We see the observed rates of convergence are for small $\sigma$ often better than the theoretical rates and furthermore the observed $L_1$ and $L_2$ rates follow closely those predicted in \cite{HefterHerzwurm2018,HefterJentzen2019,HHMG2019}. 
Also illustrated on both (a) and (b) are the positions of key $\sigma$ values from Table \ref{tab:sigmavals}, see also Table \ref{tab:comparetheory} for predicted theoretical rates.
We note from Fig.~\ref{fig:rateM1000} that in the parameter regime $\sigma\in(0,0.2)$ we observe rate $1$ rather than the predicted theoretical rate of $1/4$.
Where $\alpha=0$ (at $\sigma=0.4$) the  splitting/adaptive scheme as well as {\tt Fully Truncated} and {\tt Implicit} appear to have rate $\approx 1/2$, where as {\tt Projected}  has a lower rate. Then for $\alpha<0$ {\tt Projected} seems to have a higher rate than the others (e.g. at $\sigma=0.6,0.8$).
However Fig.~\ref{fig:rateM1000} only considers the rate of convergence, and not the size of the error constant. 

\begin{figure}
\centering
\begin{subfigure}[b]{0.49\textwidth}
\includegraphics[width=\textwidth]{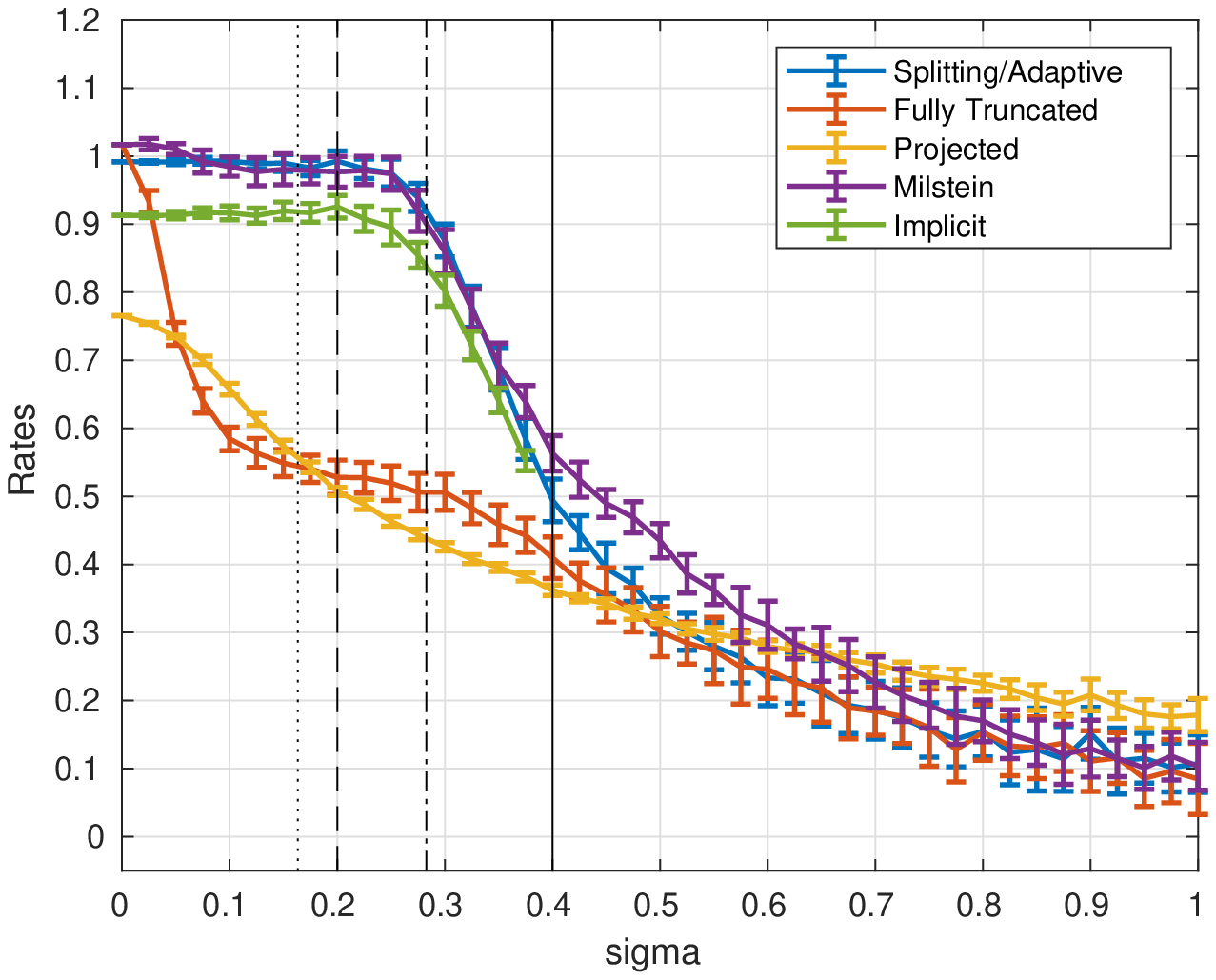}
\caption{$L_1$ error}
\end{subfigure}
\begin{subfigure}[b]{0.49\textwidth}
\includegraphics[width=\textwidth]{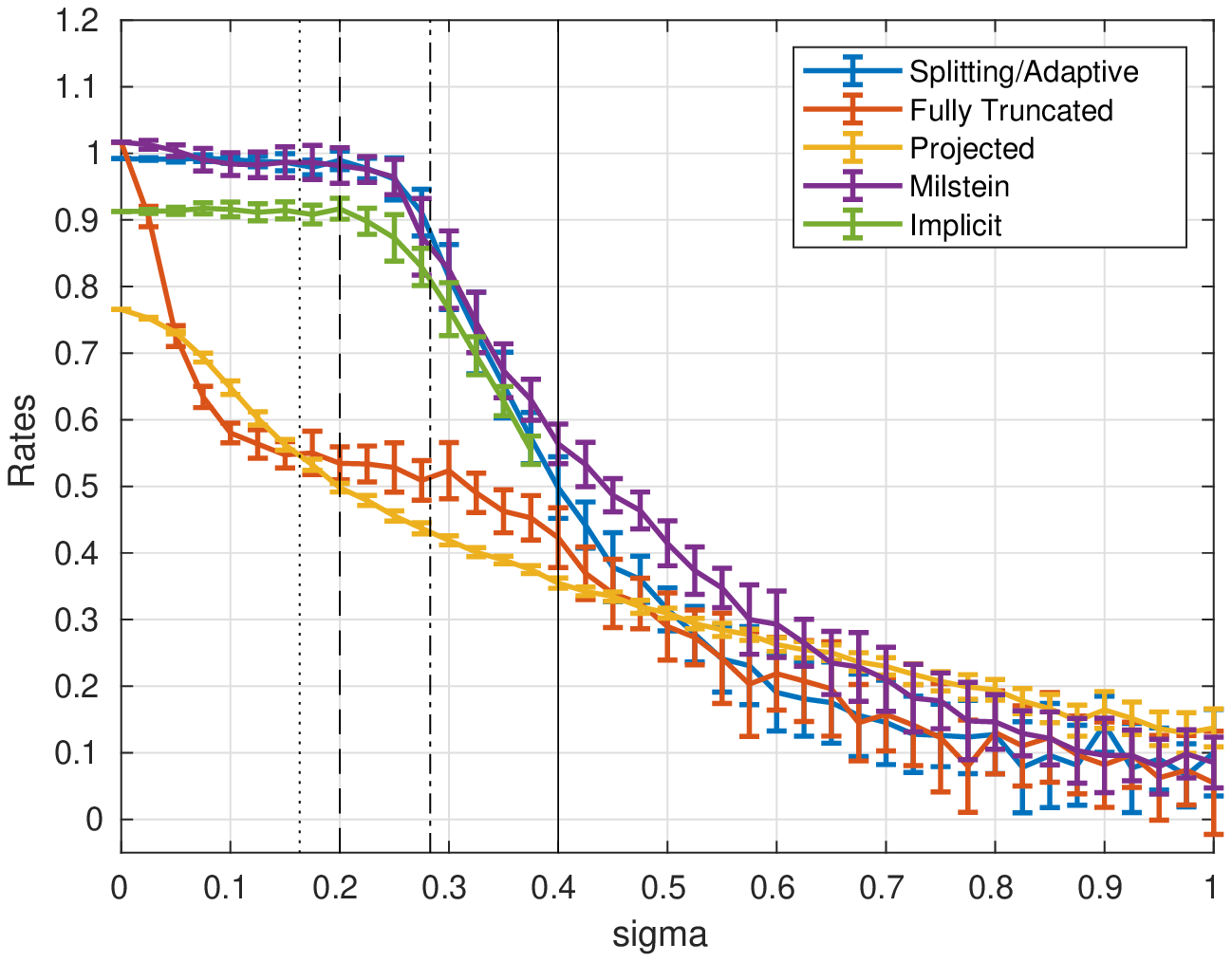}
\caption{$L_2$ error}
\end{subfigure}
\caption{Comparison of rate of convergence for our method against {\tt Implicit}, {\tt Fully Truncated}, {\tt Projected} and {\tt Milstein} with $1000$ samples, taking $20$ groups of $50$ realizations to estimate the standard deviation in the error bars. In (a) $L_1$ error and (b) $L_2$ error.
}
\label{fig:rateM1000}
\end{figure}
In Fig. \ref{fig:compareconv} we 
present the convergence plots for the $L_1$ and $L_2$ errors at $\sigma=0.3$ (such plots are the basis of the rates presented in  Fig. \ref{fig:rateM1000}). 
We see that {\tt Projected} has the largest error constant and this remains true for larger  values of sigma.  Although {\tt Projected} shows a good rate of convergence for large $\sigma$ in Fig. \ref{fig:rateM1000}, the error constant is the largest of the methods. 
In Fig \ref{fig:compareconv} we also include another adaptive version of the splitting method, this time denoted {\tt Adaptive Splitting}. 
Since for $\sigma=0.3$ we have $\alpha>0$ there is no need for adaptivity (no need for \emph{soft-zero}) for the scheme to be well defined, but we do observe an improvement in performance as measured by the size of the error constant if adaptivity is used. To examine this effect we take the heuristic choice 
$$\Dt_{n+1}=\Dtmax/(1+3\exp(-150 X_n)),$$
then for small $X_n$ the time step is smaller and $\Dt_{n+1}$ asymptotes to $\Dtmax$ as $X_n$ approaches the mean value $\theta$. 
We observe this adaptive method has the best error constant when compared to the fixed step methods we examine here. However, we have not examined adaptivity for the other schemes and it is an open question what an optimal timestepping strategy might be for a given $\sigma$.

\begin{figure}
\centering
\begin{subfigure}[b]{0.49\textwidth}
\includegraphics[width=\textwidth]{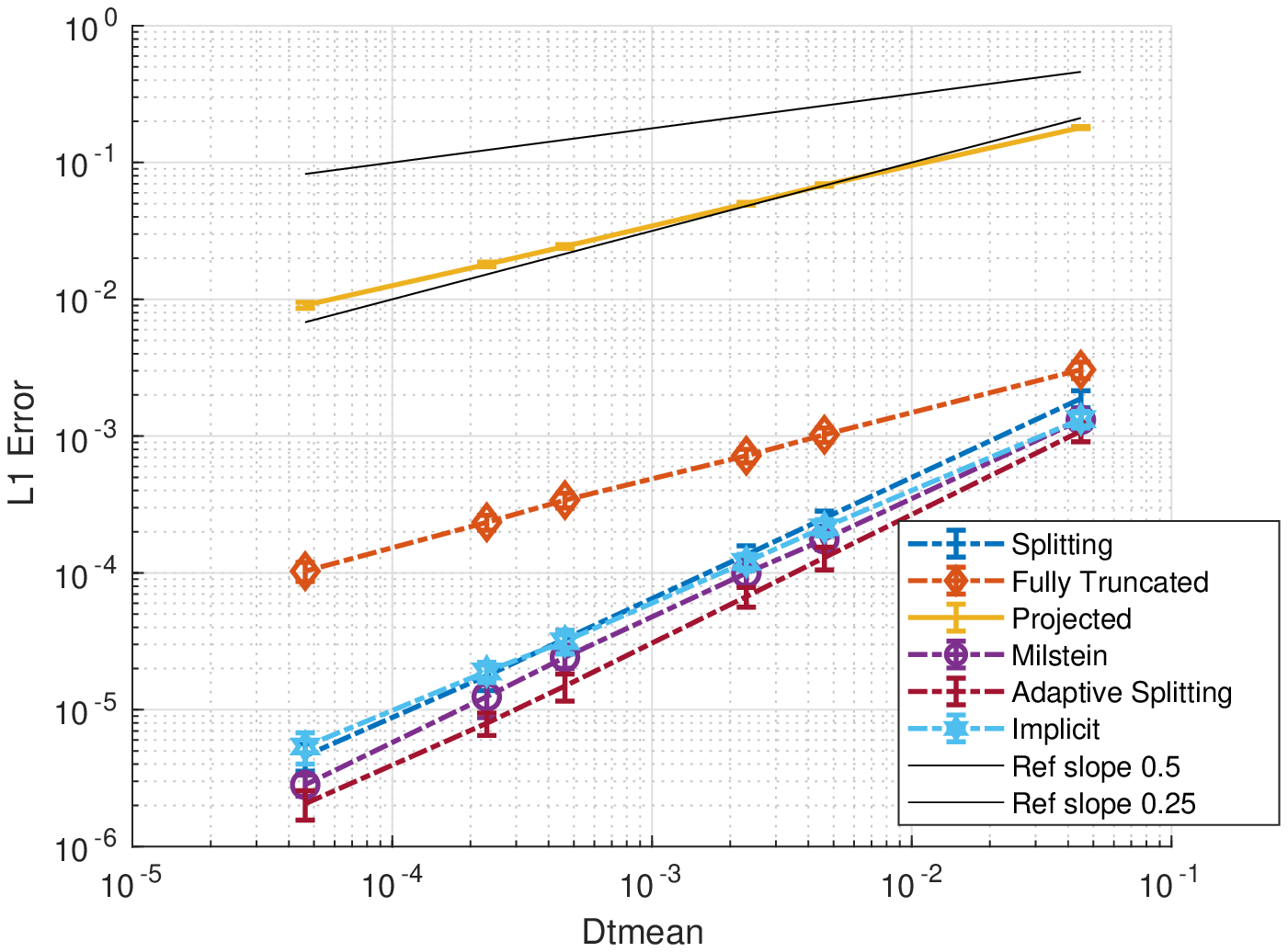}
\caption{}
\end{subfigure}
\begin{subfigure}[b]{0.49\textwidth}
\includegraphics[width=\textwidth]{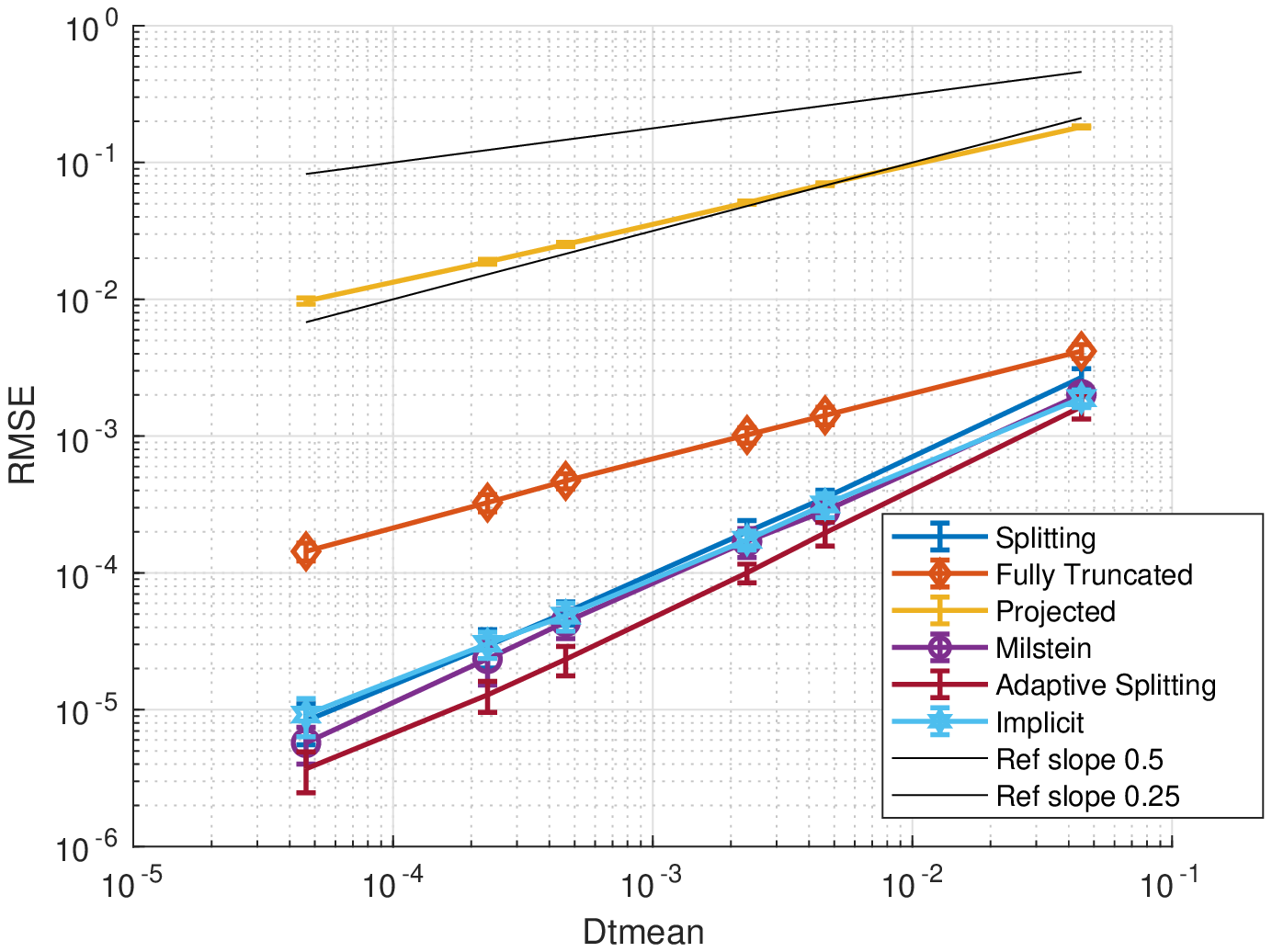}
\caption{}
\end{subfigure}
\caption{Comparison of convergence in $L_1$ (a) and $L_2$ (b) for $\sigma=0.3$.
{\tt Adaptive splitting} yields an improvement in the error constant.}
\label{fig:compareconv}
\end{figure}

\section*{Acknowledgement}
The authors wish to thank the anonymous referees for their careful reading of the manuscript and helpful suggestions.

\bibliographystyle{plain}
\bibliography{CIRb}
\end{document}